\documentclass[11pt, a4paper, reqno]{amsart}
\newcommand{\N}{\mathbb{N}}

\usepackage[utf8]{inputenc}
\usepackage{amscd}
\usepackage{amssymb}
\usepackage{newunicodechar}
\usepackage[utf8]{inputenc}
\usepackage{amsthm}
\usepackage[table]{xcolor}
\usepackage{array}
\usepackage{listings}
\usepackage[colorlinks=true,urlcolor=black]{hyperref}  
\usepackage{tikz}

\usetikzlibrary{calc}

\usepackage{amsmath}
\usepackage{graphicx,enumitem}
\setlist[enumerate]{topsep=0pt, itemsep=0pt, parsep=0pt, partopsep=0pt}
\usepackage{latexsym}

\usepackage{hyperref}
\usepackage[OT1]{fontenc}
\usepackage{bm}
\usepackage{upref}
\usepackage{color}
\usepackage{xcolor}
\usepackage{subcaption}
\usepackage[utf8]{inputenc}
\usepackage[T1]{fontenc}
\usepackage{adjustbox}
\usepackage{tikz}
\usepackage{tipa}
\usepackage[OT1]{fontenc}
\usetikzlibrary{calc}
\usepackage[margin=1in]{geometry}

\usetikzlibrary{
    arrows.meta,
    calc,
    shadows.blur
}
\usetikzlibrary{decorations.pathreplacing,positioning}
\usepackage{tikz-cd}
\usepackage{listings}
\usepackage{xcolor}
\newtheorem{question}{Question}
\usepackage[T1]{fontenc}
\lstdefinelanguage{SageMath}{
  language=Python,
  morekeywords={
    GF, PolynomialRing, vector_space_dimension,
    ideal, gens, hom
  },
  sensitive=true,
  morecomment=[l]{\#},
  morestring=[b]",
  morestring=[b]'
}

\theoremstyle{plain}
\newtheorem{theorem}{Theorem}[section]
\newtheorem{prop}[theorem]{Proposition}
\newtheorem{lemma}[theorem]{Lemma}
\newtheorem{corollary}[theorem]{Corollary}

\theoremstyle{definition}
\newtheorem{defi}[theorem]{Definition}
\newtheorem{ex}[theorem]{Example}

\theoremstyle{remark}
\newtheorem{remark}[theorem]{Remark}

\usepackage{booktabs}
 \usepackage{mathrsfs}
 
\theoremstyle{plain}

\numberwithin{equation}{section}
\allowdisplaybreaks
\setlist{leftmargin=15pt,labelindent=15pt}
\setlist[enumerate]{wide=0pt, leftmargin=15pt, labelwidth=15pt, align=left}
\setlist[enumerate]{label=(\roman*),topsep=0pt,itemsep=2pt,parsep=0pt}
\newcommand{\Fq}{\mathbb{F}_q}

\DeclareMathOperator{\Par}{Par}
\newcommand{\C}{\mathscr{C}}

\setlist[enumerate]{label=(\roman*)}

\makeatletter
\renewcommand{\subsection}{\@startsection{subsection}{2}{\z@}%
  {-3.25ex \@plus -1ex \@minus -.2ex}%
  {1.5ex \@plus .2ex}%
  {\normalfont\bfseries}}
\makeatother

\title[Affine Varieties over Finite Fields and Polygonal Number Sequences]{Rational Points on Affine Varieties over Finite Fields and Polygonal Number Sequences}

\author[Sakshi Gupta]{Sakshi Gupta}
\author[Anit Kuckian]{Anit Kuckian}
\author[Indranath Sengupta]{Indranath Sengupta$^{*}$}
\address[] {\newline
Indian Institute of Technology Gandhinagar, 
    Palaj, Gandhinagar, Gujarat 382355, India}
\email[] {sakshi.gupta2693@gmail.com, anit.kuckian@alumni.iitgn.ac.in, indranathsg@iitgn.ac.in}
\thanks{$^{*}$Corresponding Author}

\begin{document}

\vspace{-9mm}
\begin{abstract}
Authors, in an earlier work, proved that prime power subsequences of the polygonal numbers corresponding to $4$, $6$, $8$, and $10$ arise naturally as the numbers of $\mathbb{F}_q$-rational points of certain double Danielewski surfaces. In this work, authors generalize the previous construction and give a new family of affine varieties over finite fields such 
that the number of $\mathbb{F}_q$-rational points of these varieties give the prime power subsequence of every polygonal number sequence.
\end{abstract}

\maketitle

\section{Introduction}

Polygonal numbers form one of the classical families of figurate numbers,
obtained by arranging points in polygonal patterns; see, for example,
\cite{figurate}. Triangular, square, pentagonal, and hexagonal numbers are the
first familiar cases. The polygonal-number sequences are indexed by integers $m\geq3$, with $m=3,4,5,$ and $6$ corresponding to the triangular, square, pentagonal, and hexagonal numbers, respectively. 
Centered polygonal numbers arise from a related
construction in which a central point is surrounded by successive polygonal
layers. Throughout this paper, the term \emph{polygonal numbers} refers to the
classical polygonal numbers; centered polygonal numbers will always be
referred to explicitly. Although these sequences have a geometric and
combinatorial origin, their arithmetic properties have also been studied in
number theory.

\medskip
The present work is motivated by a connection between polygonal numbers and
double Danielewski surfaces over finite fields. These surfaces arose in
the study of the cancellation problem in affine algebraic geometry \cite{Gupta2019}.

\medskip
Let $q$ be a prime power, let $\Fq$ denote the finite field with $q$
elements, and let $\Fq^{\times}=\Fq\setminus\{0\}$. We write
$\mathbb N=\{1,2,\ldots\}$. Over $\Fq$, a double Danielewski surface is
defined by equations
$$
X^{d_1}Y=P(X,Z),
\qquad
X^{d_2}T=Q(X,Y,Z),
$$
where $d_1,d_2\in\mathbb N$, $P\in\Fq[X,Z]$, and
$Q\in\Fq[X,Y,Z]$, subject to suitable degree and monicity conditions.

\medskip
In our earlier work \cite{ddsfq}, we studied rational points on Danielewski
and double Danielewski surfaces over finite fields. 
For a double Danielewski surface, the solutions with nonzero $X$-coordinate contribute $q(q-1)$ points, while the solutions with $X=0$ are counted using the simultaneous zeros of $P(0,Z)$ and $Q(0,Y,Z)$.
For suitable choices of
$P$ and $Q$, Theorem~7.2 of \cite{ddsfq} gives example of double Danielewski surfaces
with
$$
q^2,\qquad
2q^2-q,\qquad
3q^2-2q,\qquad
4q^2-3q
$$
$\Fq$-rational points. These are precisely the square, hexagonal, octagonal,
and decagonal numbers evaluated at $q$. This observation leads to the following question.

\begin{question}
Can one construct polynomial families whose $\Fq$-rational point counts realize the prime-power subsequences of polygonal numbers and, separately, of centered polygonal numbers?
\end{question}

\medskip
The aim of this paper is to answer this question by introducing a familly of polynomial systems based on the form of the double Danielewski surfaces. For each such system, setting
$X=0$ in the polynomials gives an associated
polynomial family. If $N_0$ denotes the number of common zeros of this
associated family over $\Fq$, then the rational-point count has the form
$$
q(q-1)+qN_0.
$$
Hence it is enough to construct associated families with the required values
of $N_0$.

\medskip
The four polygonal-number point counts obtained in \cite{ddsfq} motivate the initial cases $m=4,6,8,$ and $10$ in our construction. We construct new associated families for these cases and use them to obtain the remaining even cases recursively. The odd cases are obtained by adding one more equation. We then use the same method for centered polygonal numbers, with different initial associated families and an extra step to recover the central point. In this way, we obtain the prime-power subsequences of polygonal and centered polygonal numbers from explicit polynomial families.

\medskip
The paper is organized as follows: In Section~\ref{sec:preliminaries}, we
recall polygonal and centered polygonal numbers and define their prime-power
subsequences. In Section~\ref{sec:point-counting-framework}, we introduce the
general polynomial families used throughout the paper and prove the main
point-count formula (Theorem~\ref{thm q(q-1)+qN_0}). We also show that, for a family of $n$ polynomials, the value of $m$ is bounded above in terms of $n$ (Corollary~\ref{cor:polygonal-bound}).

\medskip
In Section~\ref{sec:even-polygonal-surfaces}, we construct the associated
families recursively from the four initial cases corresponding to $m=4,6,8,$ and $10$, and
obtain the prime-power subsequences for all even polygonal numbers
(Theorem~\ref{thm:even-polygonal-realization}).
Section~\ref{sec:odd-polygonal-surfaces} gives the corresponding construction
for odd polygonal numbers using an additional equation
(Theorem~\ref{thm:odd-polygonal-realization}). In
Section~\ref{sec:centered-polygonal-surfaces}, we consider centered polygonal numbers, first without the central point for even and odd $m$
(Theorems~\ref{thm:centered-even-realization} and
\ref{thm:centered-odd-realization}), and then recover the full centered
polygonal numbers (Theorem~\ref{thm:centered-polygonal-realization}).

\medskip
As the recursive families become larger for larger values of $m$,
Section~\ref{sec:symbolic-codes} introduces symbolic codes for writing them
in a compact form. Finally, Section~\ref{sec:computational-resources} gives
SageMath \cite{sagemath} codes for constructing the polynomial families
and Macaulay2 \cite{M2} code for checking the corresponding point counts. Finally, Section~\ref{sec:further-directions} gives some possible directions for further study.

\section{Preliminaries}
\label{sec:preliminaries}

\subsection{Polygonal numbers}

\begin{defi}
Let $m\geq3$ and $n\geq1$ be integers. The $n$-th $m$-gonal
number is
$$
P_m(n)
=
\sum_{k=0}^{n-1}\bigl(1+k(m-2)\bigr).
$$
Equivalently,
$$
P_m(n)
=
\frac{n\bigl((m-2)n-m+4\bigr)}{2}.
$$
\end{defi}

The triangular and square cases are illustrated in
Figures~\ref{fig:triangular-numbers}
and~\ref{fig:square-numbers}, respectively.

\begin{figure}[htbp]
\centering

\begin{tikzpicture}[scale=0.67,transform shape]

\colorlet{tline}{gray!70}
\def\d{1.12}
\def\h{0.96}
\def\dotsize{6pt}

\newcommand{\OddTriBall}[3]{%
  \node[
    circle,
    shade,
    ball color=#1,
    inner sep=0pt,
    minimum size=\dotsize
  ] at ({#2},{#3}) {};
}

\newcommand{\OddTriNumber}[3]{%
  \begin{scope}[shift={(#1,0)}]
    \ifnum#2=1
      \OddTriBall{red}{0}{0}
      \node at (0,-2.0) {#3};
    \else
      \pgfmathtruncatemacro{\N}{#2}

      \foreach \m in {\N,...,2} {
        \pgfmathtruncatemacro{\depth}{\N-\m+1}
        \pgfmathtruncatemacro{\j}{\N-\m}
        \pgfmathtruncatemacro{\Mm}{\m-1}

        \def\shellcolor{red}
        \ifnum\depth=1\relax
          \def\shellcolor{red}
        \fi
        \ifnum\depth=2\relax
          \def\shellcolor{blue}
        \fi
        \ifnum\depth=3\relax
          \def\shellcolor{yellow!85!orange}
        \fi
        \ifnum\depth=4\relax
          \def\shellcolor{green!70!black}
        \fi
        \ifnum\depth=5\relax
          \def\shellcolor{orange}
        \fi

        \ifnum\m>1
          \foreach \k
          [evaluate=\k as \kp using int(\k+1)]
          in {0,...,\numexpr\m-2\relax} {
            \draw[tline,line width=1pt]
              ({\d*\j},{(\Mm/2-\k)*\h})
              --
              ({\d*\j},{(\Mm/2-\kp)*\h});
          }
        \fi

        \foreach \k in {0,...,\Mm} {
          \OddTriBall{\shellcolor}
          {\d*\j}
          {(\Mm/2-\k)*\h}
        }

        \draw[tline,line width=1pt]
          ({\d*\j},{(\m-1)/2*\h})
          --
          ({\d*(\j+1)},{(\m-2)/2*\h});

        \draw[tline,line width=1pt]
          ({\d*\j},{-(\m-1)/2*\h})
          --
          ({\d*(\j+1)},{-(\m-2)/2*\h});
      }

      \OddTriBall{orange}{\d*(\N-1)}{0}

      \node at ({\d*(\N-1)/2},-2.0) {#3};
    \fi
  \end{scope}
}

\OddTriNumber{0.0}{1}{$P_3(1)=1$}
\OddTriNumber{1.8}{2}{$P_3(2)=3$}
\OddTriNumber{4.4}{3}{$P_3(3)=6$}
\OddTriNumber{8.0}{4}{$P_3(4)=10$}
\OddTriNumber{13.5}{5}{$P_3(5)=15$}

\end{tikzpicture}

\vspace{-0.3cm}

$$
\vspace{-0.3cm}
P_3(n)=\frac{n(n+1)}{2}.
$$

\caption{The first five triangular numbers.}
\label{fig:triangular-numbers}
\vspace{0.5em}
\end{figure}
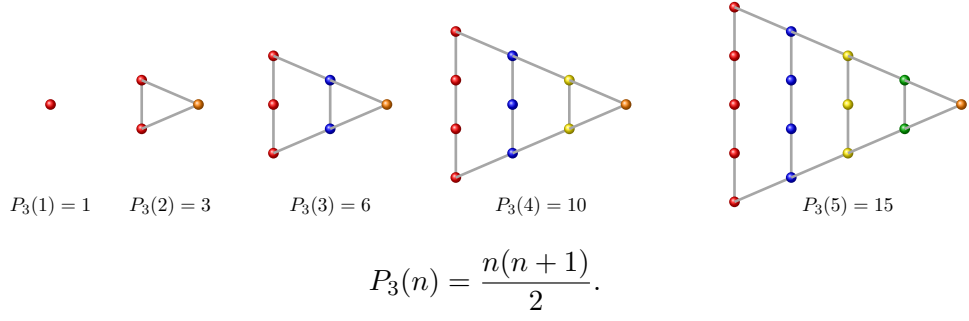

\begin{figure}[htbp]
\centering

\begin{tikzpicture}[scale=0.75,transform shape]

\colorlet{sline}{gray!70}
\def\dotsize{6pt}

\newcommand{\SquareBall}[3]{%
  \node[
    circle,
    shade,
    ball color=#1,
    inner sep=0pt,
    minimum size=\dotsize
  ] at #2 {};
}

\newcommand{\SquareNestThreeD}[3]{%
  \begin{scope}[shift={(#1,0)}]
    \ifnum#2=1
      \SquareBall{red}{(0,0)}{}
      \node at (0,-1.9) {#3};
    \else
      \coordinate (L0) at (0,0);
      \coordinate (T0) at (1.35,1.25);
      \coordinate (R0) at (2.70,0);
      \coordinate (B0) at (1.35,-1.25);

      \pgfmathtruncatemacro{\N}{#2}

      \foreach \m in {\N,...,2} {
        \pgfmathsetmacro{\s}{(\m-1)/(\N-1)}
        \pgfmathtruncatemacro{\depth}{\N-\m+1}

        \def\shellcolor{red}
        \ifnum\depth=1\relax
          \def\shellcolor{red}
        \fi
        \ifnum\depth=2\relax
          \def\shellcolor{blue}
        \fi
        \ifnum\depth=3\relax
          \def\shellcolor{yellow!85!orange}
        \fi
        \ifnum\depth=4\relax
          \def\shellcolor{green!70!black}
        \fi
        \ifnum\depth=5\relax
          \def\shellcolor{orange}
        \fi

        \coordinate (L) at ($(R0)!\s!(L0)$);
        \coordinate (T) at ($(R0)!\s!(T0)$);
        \coordinate (R) at (R0);
        \coordinate (B) at ($(R0)!\s!(B0)$);

        \draw[sline,line width=1pt]
        (L)--(T)--(R)--(B)--cycle;

        \pgfmathtruncatemacro{\mm}{\m-2}

        \foreach \P/\Q in {L/T,T/R,R/B,B/L} {
          \foreach \k in {0,...,\mm} {
            \pgfmathsetmacro{\t}{\k/(\m-1)}
            \SquareBall{\shellcolor}
            {($(\P)!\t!(\Q)$)}
            {}
          }
        }
      }

      \node at (1.35,-1.9) {#3};
    \fi
  \end{scope}
}

\SquareNestThreeD{0.0}{1}{$P_4(1)=1$}
\SquareNestThreeD{1.8}{2}{$P_4(2)=4$}
\SquareNestThreeD{6.2}{3}{$P_4(3)=9$}
\SquareNestThreeD{10.1}{4}{$P_4(4)=16$}
\SquareNestThreeD{14.2}{5}{$P_4(5)=25$}

\end{tikzpicture}

\vspace{-0.3cm}

$$
P_4(n)=n^2.
\vspace{-0.3cm}
$$

\caption{The first five square numbers.}
\label{fig:square-numbers}
\end{figure}
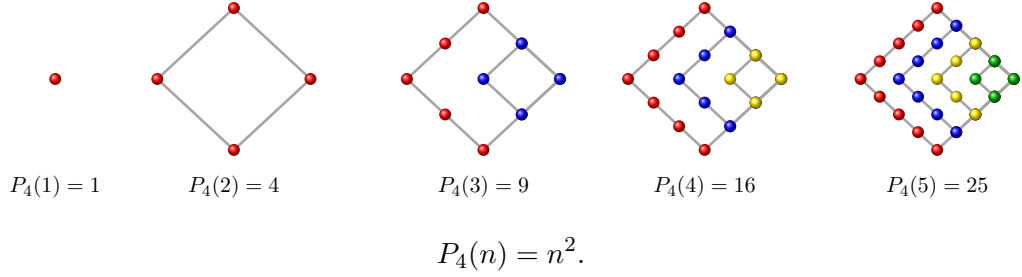

Let
$$
\mathcal{Q}
=
\left\{
q:
q=p^r
\text{ for some prime }p
\text{ and integer }r\geq1
\right\}
$$
be the set of all prime powers, written in increasing order as
$$
\mathcal{Q}
=
\{q_1<q_2<q_3<\cdots\}.
$$

\begin{defi}
For a fixed integer $m\geq3$, prime-power subsequence of the
$m$-gonal numbers is
$$
(\mathcal{P}_m(q))_{q\in \mathcal{Q}}
=
\bigl(
P_m(q_1),
P_m(q_2),
P_m(q_3),
\ldots
\bigr).
$$

\end{defi}

\subsection{Centered polygonal numbers}


\begin{defi}\label{def:centered-polygonal-number}
Let $m\geq3$ and $n\geq1$ be integers. The $n$-th centered
$m$-gonal number is
$$
C_m(n)
=
1+m+2m+\cdots+(n-1)m.
$$
Equivalently,
$$
C_m(n)
=
1+\frac{mn(n-1)}{2}.
$$
\end{defi}

This construction is illustrated for centered triangular and centered
square numbers in Figures~\ref{fig:centered-triangular-numbers}
and~\ref{fig:centered-square-numbers}, respectively.

\begin{figure}[htbp]
\centering

\begin{tikzpicture}[scale=0.72,transform shape]

\colorlet{ctline}{gray!70}
\def\ctdotsize{6pt}

\newcommand{\CenteredTriBall}[2]{%
  \node[
    circle,
    shade,
    ball color=#1,
    inner sep=0pt,
    minimum size=\ctdotsize
  ] at #2 {};
}

\newcommand{\CenteredTriShell}[4]{%
  \coordinate (#1A) at (90:#2);
  \coordinate (#1B) at (210:#2);
  \coordinate (#1C) at (330:#2);

  \draw[ctline,line width=1pt]
  (#1A)--(#1B)--(#1C)--cycle;

  \pgfmathtruncatemacro{\ctlast}{#3-1}

  \foreach \P/\Q in {A/B,B/C,C/A} {
    \foreach \j in {0,...,\ctlast} {
      \pgfmathsetmacro{\t}{\j/#3}
      \CenteredTriBall{#4}
      {($(#1\P)!\t!(#1\Q)$)}
    }
  }
}

\newcommand{\CenteredTriNumber}[3]{%
  \begin{scope}[shift={(#1,0)}]

    \pgfmathtruncatemacro{\N}{#2}

    \CenteredTriBall{orange}{(0,0)}

    \ifnum\N>1
      \pgfmathtruncatemacro{\Nm}{\N-1}

      \foreach \k in {1,...,\Nm} {
        \pgfmathtruncatemacro{\depth}{\N-\k}

        \def\shellcolor{red}
        \ifnum\depth=1\relax
          \def\shellcolor{red}
        \fi
        \ifnum\depth=2\relax
          \def\shellcolor{blue}
        \fi
        \ifnum\depth=3\relax
          \def\shellcolor{yellow!85!orange}
        \fi
        \ifnum\depth=4\relax
          \def\shellcolor{green!70!black}
        \fi
        \ifnum\depth=5\relax
          \def\shellcolor{violet}
        \fi

        \pgfmathsetmacro{\radius}{0.58*\k}

        \CenteredTriShell
        {CT#2S\k}
        {\radius}
        {\k}
        {\shellcolor}
      }
    \fi

    \node at (0,-1.85) {#3};

  \end{scope}
}

\CenteredTriNumber{0.0}{1}{$C_3(1)=1$}
\CenteredTriNumber{1.8}{2}{$C_3(2)=4$}
\CenteredTriNumber{4.8}{3}{$C_3(3)=10$}
\CenteredTriNumber{9.0}{4}{$C_3(4)=19$}
\CenteredTriNumber{14.6}{5}{$C_3(5)=31$}

\end{tikzpicture}

\vspace{-0.3cm}

$$
\vspace{-0.3cm}
C_3(n)
=
1+\frac{3n(n-1)}{2}.
$$

\caption{The first five centered triangular numbers.}
\label{fig:centered-triangular-numbers}
\end{figure}
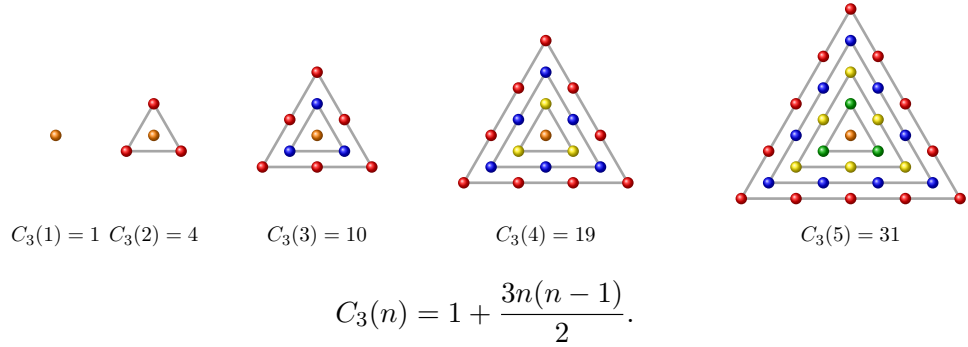

\begin{figure}[htbp]
\centering

\begin{tikzpicture}[scale=0.68,transform shape]

\colorlet{csline}{gray!70}
\def\csdotsize{6pt}

\newcommand{\CenteredSquareBall}[2]{%
  \node[
    circle,
    shade,
    ball color=#1,
    inner sep=0pt,
    minimum size=\csdotsize
  ] at #2 {};
}

\newcommand{\CenteredSquareShell}[4]{%
  \coordinate (#1T) at (90:#2);
  \coordinate (#1R) at (0:#2);
  \coordinate (#1B) at (270:#2);
  \coordinate (#1L) at (180:#2);

  \draw[csline,line width=1pt]
  (#1T)--(#1R)--(#1B)--(#1L)--cycle;

  \pgfmathtruncatemacro{\cslast}{#3-1}

  \foreach \P/\Q in {T/R,R/B,B/L,L/T} {
    \foreach \j in {0,...,\cslast} {
      \pgfmathsetmacro{\t}{\j/#3}
      \CenteredSquareBall{#4}
      {($(#1\P)!\t!(#1\Q)$)}
    }
  }
}

\newcommand{\CenteredSquareNumber}[3]{%
  \begin{scope}[shift={(#1,0)}]

    \pgfmathtruncatemacro{\N}{#2}

    \CenteredSquareBall{orange}{(0,0)}

    \ifnum\N>1
      \pgfmathtruncatemacro{\Nm}{\N-1}

      \foreach \k in {1,...,\Nm} {
        \pgfmathtruncatemacro{\depth}{\N-\k}

        \def\shellcolor{red}
        \ifnum\depth=1\relax
          \def\shellcolor{red}
        \fi
        \ifnum\depth=2\relax
          \def\shellcolor{blue}
        \fi
        \ifnum\depth=3\relax
          \def\shellcolor{yellow!85!orange}
        \fi
        \ifnum\depth=4\relax
          \def\shellcolor{green!70!black}
        \fi
        \ifnum\depth=5\relax
          \def\shellcolor{violet}
        \fi

        \pgfmathsetmacro{\radius}{0.55*\k}

        \CenteredSquareShell
        {CS#2S\k}
        {\radius}
        {\k}
        {\shellcolor}
      }
    \fi

    \node at (0,-3.0) {#3};

  \end{scope}

}
\CenteredSquareNumber{0.0}{1}{$C_4(1)=1$}
\CenteredSquareNumber{2.2}{2}{$C_4(2)=5$}
\CenteredSquareNumber{5.8}{3}{$C_4(3)=13$}
\CenteredSquareNumber{10.6}{4}{$C_4(4)=25$}
\CenteredSquareNumber{16.8}{5}{$C_4(5)=41$}

\end{tikzpicture}

\vspace{-0.3cm}

$$
\vspace{-0.3cm}
C_4(n)
=
1+2n(n-1)
=
n^2+(n-1)^2.
$$

\caption{The first five centered square numbers.}
\label{fig:centered-square-numbers}
\end{figure}

The prime-power subsequence  of the centered $m$-gonal numbers is
$$
(\mathcal{C}_m(q))_{q \in \mathcal{Q}}
=
\bigl( C_m(q_1), C_m(q_2), C_m(q_3), \ldots \bigr).
$$

\section{A general point-counting framework}
\label{sec:point-counting-framework}

We now introduce a general family of polynomial systems whose $\mathbb F_q$-rational points will be used to realize the prime power subsequences of polygonal and centered polygonal numbers.
\medskip

Let $q$ be a prime power and let $n\geq1$. For $1\leq i\leq n$, let
$
H_i\in\mathbb F_q[X,Y_0,\ldots,Y_{i-1}]
$
and let $d_i$ be a positive integer. Define
\begin{equation}
G_i
=
X^{d_i}Y_i-H_i(X,Y_0,\ldots,Y_{i-1}),
\qquad
1\leq i\leq n,
\label{eq:defining-polynomials}
\end{equation}
and let
$
\mathcal G=(G_1,\ldots,G_n)
$ denote the family of $n$ polynomials. 

Let $\mathcal F=(F_1,\ldots,F_n)$ be the family obtained from
$\mathcal G$ by evaluating each $H_i$ at $X=0$, where
\begin{equation}
F_i(Y_0,\ldots,Y_{i-1})
=
H_i(0,Y_0,\ldots,Y_{i-1}),
\qquad 1\leq i\leq n.
\label{eq:associated-family}
\end{equation}
Suppose $H_{i} = A_{0} + A_{1}X + \cdots + A_{e}X^{e}$, then $F_{i} = A_{0}$ and 
it follows that $H_i-F_i$ is divisible by $X$. We may write
$
H_i=F_i+XR_i
$, 
for some
$
R_i\in\mathbb F_q[X,Y_0,\ldots,Y_{i-1}]. 
$
We call $\mathcal F$ the \textit{associated family} of $\mathcal G$, where 
\begin{equation}
G_i
=
X^{d_i}Y_i
-
F_i(Y_0,\ldots,Y_{i-1})
-
XR_i(X,Y_0,\ldots,Y_{i-1}),
\qquad
1\leq i\leq n.
\label{eq:Gi-FR}
\end{equation}

Conversely, given such a family $\mathcal F$,
choosing positive integers $d_i$ and polynomials $R_i$ in
\eqref{eq:Gi-FR} gives a family $\mathcal G$ having
$\mathcal F$ as its associated family.
Clearly, the choice of $\mathcal G$ is not unique.

\smallskip
Throughout the paper, capital letters
$X,Y_0,Y_1,Y_2,\ldots$ denote variables, and the corresponding
lowercase letters $x,y_0,y_1,y_2,\ldots$ denote their values in
$\mathbb F_q$.

\smallskip
We associate to $\mathcal G$ the set
\begin{equation}
S_{\mathcal G}(\mathbb F_q)
=
\left\{
(x,y_0,y_1,\ldots,y_n)\in\mathbb F_q^{n+2}
:
G_i(x,y_0,y_1,\ldots,y_i)=0
\text{ for all }1\leq i\leq n
\right\} \subseteq \Fq^{n+2}.
\label{eq:SG}
\end{equation}

Note that, $G_i(0,y_0,y_1,\ldots,y_i)= F_i(y_0,y_1,\ldots,y_{i-1})$, $1\leq i\leq n.
$
\medskip

For $n=2$, if $H_1(X,Y_0)$ is monic in $Y_0$ with
$\deg_{Y_0}H_1\geq2$, and $H_2(X,Y_0,Y_1)$ is monic in $Y_1$ with
$\deg_{Y_1}H_2\geq2$, then the polynomial family
$
\left(
X^{d_1}Y_1-H_1(X,Y_0),\,
X^{d_2}Y_2-H_2(X,Y_0,Y_1)
\right)
$
defines a double Danielewski surface. Thus, we call $\mathcal{G} = (G_1,\ldots,G_n)$ the {$m$-gonal Danielewski variety.}


\begin{theorem}\label{thm q(q-1)+qN_0}
Let $ \mathcal{G}=(G_1,\ldots,G_n) $ be a family of $n$ polynomials satisfying \eqref{eq:defining-polynomials}, and let $S_{\mathcal{G}}(\mathbb{F}_q)$ be the set defined in \eqref{eq:SG}. Define $$ N_0(\mathcal{G}) = \#\left\{ (y_0,y_1,\ldots,y_{n-1})\in\mathbb{F}_q^n : F_i(y_0,y_1,\ldots,y_{i-1})=0 \text{ for all }1\leq i\leq n \right\}. $$ Then $$ \#S_{\mathcal{G}}(\mathbb{F}_q) = q(q-1)+qN_0(\mathcal{G}). $$ 
\end{theorem}

\begin{proof}
We decompose $S_{\mathcal{G}}(\mathbb{F}_q)$ into two disjoint parts
according as $x\neq0$ or $x=0$.

First suppose that
$
x\in\mathbb F_q^\times.
$
Since $x^{d_i}\neq0$, the equation
$$
y_i
=
x^{-d_i}H_i(x,y_0,\ldots,y_{i-1}),
\qquad
1\leq i\leq n.
$$
determines $y_i$ uniquely once
$
x,y_0,\ldots,y_{i-1}
$
are known. Thus, after choosing
$
x\in\mathbb F_q^\times
$
and
$
y_0\in\mathbb F_q,
$
the values $y_1,\ldots,y_n$ are determined successively.
Hence there are $q(q-1)$ points with $x\neq0$.

\medskip
Now suppose that $x=0$. Since $d_i\geq1$, the defining
equations become
$$
F_i(y_0,y_1,\ldots,y_{i-1})=0,
\qquad 1\leq i\leq n.
$$
By definition, there are
$N_0(\mathcal{G})$
choices for $(y_0,y_1,\ldots,y_{n-1})$. The variable $y_n$
is free and therefore has $q$ possible values. Hence there are
$
qN_0(\mathcal{G})
$
points with $x=0$. Combining the two cases we get $\#S_{\mathcal{G}}(\mathbb{F}_q)
=
q(q-1)+qN_0(\mathcal{G})$.
\end{proof}

\begin{corollary}
\label{cor:x-nonzero-fiber}
Under the hypotheses of Theorem~\ref{thm q(q-1)+qN_0}, the projection
$$
\left\{
(x,y_0,\ldots,y_n)\in S_{\mathcal G}(\mathbb F_q):x\neq0
\right\}
\longrightarrow
\mathbb F_q^\times\times\mathbb F_q,
$$
given by
$$
(x,y_0,\ldots,y_n)\longmapsto(x,y_0),
$$
is a bijection.
\end{corollary}

\begin{proof}
This follows from the $x\neq0$ part of the proof of
Theorem~\ref{thm q(q-1)+qN_0}.
\end{proof}
Theorem~\ref{thm q(q-1)+qN_0} shows that the point count depends on
$\mathcal G$ only through $N_0(\mathcal G)$, which is determined by its
associated family $\mathcal F=(F_1,\ldots,F_n)$. Hence different defining families with
the same associated family have the same $\mathbb F_q$-rational point
count. Thus, to obtain a prescribed point count, it is enough to
construct an associated family $\mathcal F$ with the required number
of $\Fq$-rational points.

\subsection{Bounds on $m$}

We now derive the bounds implied by Theorem~\ref{thm q(q-1)+qN_0}.
Since $N_0(\mathcal{G})$ counts a subset of $\mathbb{F}_q^n$,
$$
0\leq N_0(\mathcal{G})\leq q^n.
$$
Hence Theorem~\ref{thm q(q-1)+qN_0} gives
$$
q(q-1)
\leq
\#S_{\mathcal{G}}(\mathbb{F}_q)
\leq
q(q-1)+q^{n+1}.
$$

We now derive the restriction on $m$ imposed by the point-count bound.

\begin{corollary}\label{cor:polygonal-bound}
Suppose that, for some $m\geq3$, there exists a family
$
\mathcal G=(G_1,\ldots,G_n)
$
of $n$ polynomials as in~\eqref{eq:defining-polynomials}
such that
$$
\left(
\#S_{\mathcal G}(\mathbb F_q)
\right)_{q\in\mathcal Q}
=
\left(
P_m(q)
\right)_{q\in\mathcal Q}.
$$
Then $m$ is even and
\begin{equation}
4\leq m\leq2^{n+1}+2.
\label{eq:polygonal-bound}
\end{equation}
\end{corollary}

\begin{proof}
Since $P_m(2)=m$, Theorem~\ref{thm q(q-1)+qN_0} gives
$$
m
=
2+2N_0(\mathcal{G}).
$$
Hence $m$ is even and
$
N_0(\mathcal{G})=(m-2)/2.
$
Since $N_0(\mathcal{G})\leq2^n$, the result follows.
\end{proof}

The condition in Corollary~\ref{cor:polygonal-bound} is obtained from
the case $q=2$. The next proposition shows that, for even $m$, this
condition is sufficient for the point-count bounds to hold for every
prime power $q$.

\begin{prop}\label{prop:q2-suffices}
Let $n\geq 2$ and let $m\geq4$ be even. Then
$$
q^2
\leq
P_m(q)
\leq
q(q-1)+q^{n+1}
$$
for every prime power $q$ if and only if
$$
4\leq m\leq2^{n+1}+2.
$$
\end{prop}

\begin{proof}
For the forward implication, the case $q=2$ gives the result.

Conversely, suppose that
$$
4\leq m\leq2^{n+1}+2.
$$
Since
$$
P_m(q)
=
q^2+\frac{m-4}{2}\,q(q-1),
$$
the lower bound holds for every prime power $q$. For the upper bound,
$$
q(q-1)+q^{n+1}-P_m(q)
=
q(q-1)
\left(
1+q+\cdots+q^{n-1}-\frac{m-4}{2}
\right).
$$
Now
$$
\frac{m-4}{2}\leq2^n-1,
$$
while $q\geq2$ gives
$$
1+q+\cdots+q^{n-1}
\geq
1+2+\cdots+2^{n-1}
=
2^n-1.
$$
Hence
$$
\frac{m-4}{2}
\leq
1+q+\cdots+q^{n-1},
$$
and therefore $
P_m(q)
\leq
q(q-1)+q^{n+1}.
$
\end{proof}

\begin{question}
For a prescribed even integer $m$, what is the least integer $n \geq 2$ for which there exists a family
$
\mathcal{G}=(G_1,\ldots,G_n)
$
as in~\eqref{eq:defining-polynomials} such that
$
\#S_{\mathcal{G}}(\mathbb{F}_q)
=
P_m(q)
$
for every prime power $q$?
\end{question}

If
$
m>2^{n+1}+2,
$
then Corollary~\ref{cor:polygonal-bound} shows that a family of $n$ polynomials cannot realize $P_m(q)$ for all $q \in \mathcal{Q}$. Thus more than $n$ polynomials are necessary.

For $m'>m$,
$$
P_{m'}(q)-P_m(q)
=
\frac{m'-m}{2}\,q(q-1).
$$
By Theorem~\ref{thm q(q-1)+qN_0}, passing from $P_m(q)$ to
$P_{m'}(q)$ requires an increase of
$$
\frac{m'-m}{2}(q-1)
$$
in $N_0(\mathcal{G})$. Since adjoining only a free variable multiplies
the point count by $q$, the constructions in the next section adjoin a new variable together with a new  polynomial.

\section{Polynomial families realizing even polygonal numbers}
\label{sec:even-polygonal-surfaces}

We now construct families $\mathcal G$ such that
$\#S_{\mathcal G}(\mathbb F_q)=P_m(q)$ for all even $m\geq4$.

\subsection{Even values of $m$ and the number of polynomials}

For the recursive construction, we group the even values of $m$
according to the least number of polynomials $n \geq 2$ for which the bound
in Corollary~\ref{cor:polygonal-bound} is satisfied. More precisely, for
each even integer $m\geq4$, we consider the least integer $n\geq2$ such
that
$$
m\leq2^{n+1}+2.
$$
Thus values of $m$ with the same least integer $n$ are grouped together.

\begin{defi}\label{def:In}
For each even integer $m\geq4$, define
$$
n(m)
=
\min
\left\{
n\geq2:
m\leq2^{n+1}+2
\right\}.
$$
For $n\geq2$, define
$$
I_n
=
\left\{
m\in2\mathbb{Z}:
m\geq4,~ n(m)=n
\right\}.
$$
\end{defi}

By Corollary~\ref{cor:polygonal-bound}, if a family
$
\mathcal{G}=(G_1,\ldots,G_r)
$
of $r \geq 2$ polynomials as in~\eqref{eq:defining-polynomials}
satisfies
$
\#S_{\mathcal{G}}(\mathbb{F}_q)
=
P_m(q)
$
for every prime power $q$, then
$
r\geq n(m).
$
Thus $n(m)$ is the least integer satisfying the bound.

\begin{lemma}\label{lem:In-description}
We have
$
I_2
=
\{4,6,8,10\},
$
and, for every $n\geq3$,
$$
I_n
=
\left\{
m\in2\mathbb{Z}:
2^n+4\leq m\leq2^{n+1}+2
\right\}.
$$
\end{lemma}

\begin{proof}
For $n=2$, the condition $n(m)=2$ gives
$
4\leq m\leq10,
$
and hence
$
I_2=\{4,6,8,10\}.
$

Let $n\geq3$. By Definition~\ref{def:In}, $n(m)=n$ if and only if
$$
2^n+2<m\leq2^{n+1}+2.
$$
Since $m$ is even, this is equivalent to \,\, $
2^n+4\leq m\leq2^{n+1}+2.$
\end{proof}

The first few sets are
$$
I_2=\{4,6,8,10\},
\quad 
I_3=\{12,14,16,18\}
\quad \text{and}\quad
I_4=\{20,22,24,26,28,30,32,34\}.
$$
Further, for $n\geq3$,
$
\#I_n=2^{n-1}.$ The sets $I_n$ satisfy the following recursion.

\begin{lemma}\label{lem:interval-recursion}
For every $n\geq3$,
$$
I_n-2^n
=
I_2\cup I_3\cup\cdots\cup I_{n-1},
$$
where
$$
I_n-2^n
=
\{m-2^n:m\in I_n\}.
$$
\end{lemma}

\begin{proof}
By the definition of $I_n$,
$$
I_n-2^n
=
\{4,6,\ldots,2^n+2\}.
$$
Also,
$$
I_2\cup I_3\cup\cdots\cup I_{n-1}
=
\{4,6,\ldots,2^n+2\}.
$$
Hence the two sets are equal.
\end{proof}

\subsection{The parent relation}

For $m\in I_n$, the integer $n$ is determined by $m$.

\begin{prop}\label{prop:level-of-m}
Let $m\in I_n$ with $n\geq3$. Then
$
n
=
\left\lfloor\log_2(m-4)\right\rfloor.
$
\end{prop}

\begin{proof}
Since $m\in I_n$,
$$
2^n+4\leq m\leq2^{n+1}+2.
$$
Hence
$$
2^n\leq m-4<2^{n+1},
$$
and the result follows.
\end{proof}

Lemma~\ref{lem:interval-recursion} shows that, for each $m\in I_n$ with
$n\geq3$, there is a unique $m'\in I_2\cup\cdots\cup I_{n-1}$ such that
$$
m=m'+2^n.
$$
We denote this value $m'$ by $\Par(m)$.

\begin{defi}\label{def:parent}
Let $m\in I_n$ with $n\geq3$. The parent of $m$ is
$
\Par(m)
=
m-2^n.
$
\end{defi}

By Proposition~\ref{prop:level-of-m},
$
\Par(m)
=
m-
2^{\left\lfloor\log_2(m-4)\right\rfloor}.
$

\begin{corollary}\label{cor:parent-bijection}
For every $n\geq3$, the map
$$
\Par:I_n
\longrightarrow
I_2\cup I_3\cup\cdots\cup I_{n-1},
\qquad
m\longmapsto m-2^n,
$$
is a bijection.
\end{corollary}

\begin{proof}
By Lemma~\ref{lem:interval-recursion},
$$
I_n-2^n
=
I_2\cup I_3\cup\cdots\cup I_{n-1}.
$$
Hence every element of
$
I_2\cup\cdots\cup I_{n-1}
$
is of the form $m-2^n$ for some $m\in I_n$, so $\Par$ is surjective.

If
$
\Par(m_1)=\Par(m_2)
$
for $m_1,m_2\in I_n$, then
$
m_1-2^n=m_2-2^n,
$
and hence $m_1=m_2$. Thus $\Par$ is injective.
\end{proof}

\begin{lemma}\label{lem:polygonal-parent-difference}
Let $m\in I_n$ with $n\geq3$. Then
$$
P_m(q)-P_{\Par(m)}(q)
=
2^{n-1}q(q-1).
$$
\end{lemma}

\begin{proof}
Since
$
m-\Par(m)=2^n,
$
we have
$$
P_m(q)-P_{\Par(m)}(q)
=
\frac{m-\Par(m)}{2}\,q(q-1)
=
2^{n-1}q(q-1).
$$
\end{proof}

\subsection{Parent chains}

We now iterate the parent map.
\begin{lemma}\label{lem:parent-termination}
Let $n\geq3$ and $m\in I_n$. Repeated application of $\Par$ on $m$ reaches an element of $I_2$ after finitely many steps.
\end{lemma}

\begin{proof}
If $m\in I_n$, then Corollary~\ref{cor:parent-bijection} gives
$$
\Par(m)\in I_2\cup I_3\cup\cdots\cup I_{n-1}.
$$
If $\Par(m)\notin I_2$, then $\Par(m)\in I_r$ for some $3\leq r<n$,
and we apply $\Par$ again. At each step, the index $r$ of the set $I_r$
decreases. Hence, after finitely many steps, an element of $I_2$ is
reached.
\end{proof}

\begin{defi}\label{def:parent-chain}
Let $n\geq3$ and $m\in I_n$. By Lemma~\ref{lem:parent-termination},
there is a finite sequence
$
m_0,m_1,\ldots,m_t=m
$
such that
$
m_0\in I_2
$
and
$$
m_{i-1}=\Par(m_i),
\qquad 1\leq i\leq t.
$$
The sequence
$$
m=m_t
\longmapsto
m_{t-1}
\longmapsto
\cdots
\longmapsto
m_1
\longmapsto
m_0
$$
is called the \emph{parent chain} of $m$.
\end{defi}

Each term of the parent chain is uniquely determined by the preceding
term, since $\Par(m_i)$ is uniquely determined by $m_i$. Hence the
parent chain of $m$ is unique.

\begin{ex}\label{ex:parent-chain-358}
Let $m=358$. By Proposition~\ref{prop:level-of-m},
$$
\left\lfloor\log_2(358-4)\right\rfloor
=
\left\lfloor\log_2 354\right\rfloor
=
8,
$$
so $358\in I_8$. Hence
$$
\Par(358)=358-2^8=102.
$$
Next,
$$
\left\lfloor\log_2(102-4)\right\rfloor
=
\left\lfloor\log_2 98\right\rfloor
=
6,
$$
so $102\in I_6$ and
$$
\Par(102)=102-2^6=38.
$$
Finally,
$$
\left\lfloor\log_2(38-4)\right\rfloor
=
\left\lfloor\log_2 34\right\rfloor
=
5,
$$
so $38\in I_5$ and
$$
\Par(38)=38-2^5=6.
$$
Since $6\in I_2$, the parent chain of $358$ is
$$
358\longmapsto102\longmapsto38\longmapsto6.
$$
\end{ex}

\begin{lemma}\label{lem:parent-endpoint}
Let
$$
m=m_t\longmapsto m_{t-1}\longmapsto\cdots\longmapsto m_0
$$
be the parent chain of $m$. Then
$$
m_0\equiv m\pmod8.
$$
In particular, $m_0\in I_2$ is uniquely determined by the residue of
$m$ modulo $8$.
\end{lemma}

\begin{proof}
For $1\leq i\leq t$, if $m_i\in I_r$, then $r\geq3$ and
$$
m_{i-1}
=
\Par(m_i)
=
m_i-2^r.
$$
Since $8\mid2^r$,
$$
m_{i-1}\equiv m_i\pmod8.
$$
Hence
$$
m_0\equiv m\pmod8.
$$
Since
$$
I_2=\{4,6,8,10\},
$$
whose elements have distinct residues modulo $8$, the endpoint $m_0$
is uniquely determined.
\end{proof}
The parent chain reduces $m$ to an element $m_0\in I_2$. The polynomial
construction will proceed along the parent chain in the reverse direction.

\subsection{Block decomposition}

By Lemma~\ref{lem:parent-endpoint}, the last element of the parent chain
of $m$ is determined by $m$ modulo $8$.

\begin{lemma}\label{lem:block-representation}
Let $n\geq3$ and $m\in I_n$. Then there exist unique integers $j$ and
$k$ such that
$
1\leq j\leq2^{n-3},
~
1\leq k\leq4,
$
and
$$
m
=
2^n+4+8(j-1)+2(k-1).
$$
If $m_0\in I_2$ is the last element of the parent chain of $m$, then
$
m_0=2k+2.
$
\end{lemma}

\begin{proof}
Since $m\in I_n$,
$$
2^n+4\leq m\leq2^{n+1}+2.
$$
As $m$ is even, there exists a unique integer $r$ such that
$$
m=2^n+4+2r,
\qquad
0\leq r\leq2^{n-1}-1.
$$

By the division algorithm, there exist unique integers $u$ and $v$ such
that
$$
r=4u+v,
\qquad
0\leq v\leq3.
$$
Since $r\leq2^{n-1}-1$, we have
$$
0\leq u\leq2^{n-3}-1.
$$
Set
$$
j=u+1,
\qquad
k=v+1.
$$
Then
$$
1\leq j\leq2^{n-3},
\qquad
1\leq k\leq4,
$$
and
$$
m
=
2^n+4+8(j-1)+2(k-1).
$$
The uniqueness of $u$ and $v$ gives the uniqueness of $j$ and $k$.

Since $n\geq3$, we have $8\mid2^n$. Hence
$$
m\equiv4+2(k-1)=2k+2\pmod8.
$$
By Lemma~\ref{lem:parent-endpoint},
$$
m_0\equiv m\pmod8.
$$
Since $m_0\in\{4,6,8,10\}$, it follows that
$$
m_0=2k+2.
$$
\end{proof}

\begin{corollary}\label{cor:k-parent-chain}
Let $m\in I_r$ and $m'\in I_s$, where $r,s\geq3$, and suppose that
$m$ and $m'$ lie in the same parent chain. Write
$$
m
=
2^r+4+8(a-1)+2(k-1)
$$
and
$$
m'
=
2^s+4+8(b-1)+2(k'-1)
$$
as in Lemma~\ref{lem:block-representation}. Then
$$
k=k'.
$$
\end{corollary}

\begin{proof}
Since $m$ and $m'$ lie in the same parent chain, their parent chains
end at the same element $m_0\in I_2$. By
Lemma~\ref{lem:block-representation},
$
m_0=2k+2=2k'+2.
$
Hence
$
k=k'.
$
\end{proof}

For a fixed $j$, the four values $k=1,2,3,4$ give four consecutive even
elements of $I_n$:
$$
2^n+4+8(j-1),\quad
2^n+6+8(j-1),\quad
2^n+8+8(j-1),\quad
2^n+10+8(j-1).
$$
By Lemma~\ref{lem:block-representation}, their parent chains end at
$4,6,8,$ and $10$, respectively. This gives the following grouping of the elements of $I_n$.

\begin{defi}\label{def:blocks}
Set
$
A_{2,1}=\{4,6,8,10\}.
$
For $n\geq3$ and
$
1\leq j\leq2^{n-3},
$
define
$$
a_{n,j}=2^n+4+8(j-1)
$$
and
$$
A_{n,j}
=
\{a_{n,j},a_{n,j}+2,a_{n,j}+4,a_{n,j}+6\}.
$$
We call $A_{n,j}$ the $j$-th block of $I_n$.
\end{defi}

By Lemma~\ref{lem:block-representation},
$$
I_n
=
\bigcup_{j=1}^{2^{n-3}}A_{n,j},
$$
and these sets are disjoint.

\begin{defi}\label{def:block-chain}
Let
$$
m=m_t\longmapsto m_{t-1}\longmapsto\cdots\longmapsto m_0
$$
be the parent chain of $m$. For each $0\leq i\leq t$, let
$
m_i\in A_{n_i,j_i}.
$
The sequence
$$
A_{n_t,j_t}
\longmapsto
A_{n_{t-1},j_{t-1}}
\longmapsto
\cdots
\longmapsto
A_{n_0,j_0}
$$
is called the \emph{block chain} of $m$.
\end{defi}

\begin{ex}\label{ex:block-358}
From Example~\ref{ex:parent-chain-358},
$$
358\longmapsto102\longmapsto38\longmapsto6.
$$

For $358\in I_8$,
$$
358=2^8+4+2(49),
\qquad
49=4\cdot12+1,
$$
so $j=13$ and $k=2$. Hence
$$
358\in A_{8,13}.
$$
By Corollary~\ref{cor:k-parent-chain}, the same value $k=2$ occurs for
$102$ and $38$. Moreover,
$$
102=2^6+4+2(17),
\qquad
17=4\cdot4+1,
$$
and
$$
38=2^5+4+2(1),
\qquad
1=4\cdot0+1.
$$
Thus
$$
102\in A_{6,5},
\qquad
38\in A_{5,1}.
$$
Therefore the block chain of $358$ is
$$
A_{8,13}
\longmapsto
A_{6,5}
\longmapsto
A_{5,1}
\longmapsto
A_{2,1}.
$$
\end{ex}

The arrangement of the blocks for $2\leq n\leq6$ is shown in
Figure~\ref{fig:block-dependency-B4}. For $r<n$, the blocks
$A_{n,j}$ and $A_{r,s}$ are joined by an edge if there exists
$m\in A_{n,j}$ such that
$
\Par(m)\in A_{r,s}.
$

\begin{figure}[ht!]
\centering

\resizebox{0.90\textwidth}{!}{%
\begin{tikzpicture}[
    block/.style={
        circle,
        draw=green!35!black,
        line width=0.9pt,
        shade,
        ball color=yellow!90!orange,
        text=black,
        font=\small\bfseries,
        minimum size=13mm,
        inner sep=1pt
    },
    dependency/.style={
    draw=black!75,
    line width=0.8pt
},
    arrowlabel/.style={
    midway,
    fill=white,
    inner xsep=3pt,
    inner ysep=1pt,
    font=\large\bfseries,
    text=black
}
]

\node[block] (A1) at (7,7.6) {$A_{2,1}$};

\node[block] (A51) at (17,4.7) {$A_{6,1}$};

\node[block] (A41) at (13.8,4.7) {$A_{5,1}$};
\node[block] (A55) at (13.8,2.0) {$A_{6,5}$};

\node[block] (A31) at (10.0,4.7) {$A_{4,1}$};
\node[block] (A53) at (11.7,2.0) {$A_{6,3}$};

\node[block] (A43) at (8.15,2.0) {$A_{5,3}$};
\node[block] (A57) at (8.1,-0.75) {$A_{6,7}$};

\node[block] (A2) at (2.0,4.7) {$A_{3,1}$};

\node[block] (A52) at (4.9,2.0) {$A_{6,2}$};

\node[block] (A42) at (2.0,2.0) {$A_{5,2}$};
\node[block] (A56) at (2.0,-0.75) {$A_{6,6}$};

\node[block] (A32) at (-1.5,2.0) {$A_{4,2}$};
\node[block] (A54) at (-0.3,-0.75) {$A_{6,4}$};

\node[block] (A44) at (-2.75,-0.75) {$A_{5,4}$};
\node[block] (A58) at (-2.7,-3.40) {$A_{6,8}$};

\draw[dependency] (A2) --
node[arrowlabel] {$\Large{L}$}
(A1);
\draw[dependency] (A31) -- node[arrowlabel] {$\Large{EL}$}(A1);
\draw[dependency] (A41) -- node[arrowlabel] {$\Large{E^2L}$}(A1);
\draw[dependency] (A51) -- node[arrowlabel] {$\Large{E^3L}$}(A1);

\draw[dependency] (A32) -- node[arrowlabel] {$\Large{L}$} (A2);
\draw[dependency] (A42) -- node[arrowlabel] {$\Large{EL}$}(A2);
\draw[dependency] (A52) -- node[arrowlabel] {$\Large{E^2L}$} (A2);

\draw[dependency] (A44) -- node[arrowlabel] {$\Large{L}$} (A32);
\draw[dependency] (A54) -- node[arrowlabel] {$\Large{EL}$} (A32);
\draw[dependency] (A58) --  node[arrowlabel] {$\Large{L}$}(A44);

\draw[dependency] (A56) -- node[arrowlabel] {$\Large{L}$} (A42);

\draw[dependency] (A43) -- node[arrowlabel] {$\Large{L}$} (A31);
\draw[dependency] (A53) -- node[arrowlabel] {$\Large{EL}$} (A31);
\draw[dependency] (A57) --  node[arrowlabel] {$\Large{L}$}(A43);

\draw[dependency] (A55) -- node[arrowlabel] {$\Large{L}$} (A41);

\end{tikzpicture}%
}
\caption{ The blocks $A_{n,j}$ for $2\leq n\leq6$ and the parent relations between them.
\label{fig:block-dependency-B4}
}
\end{figure}
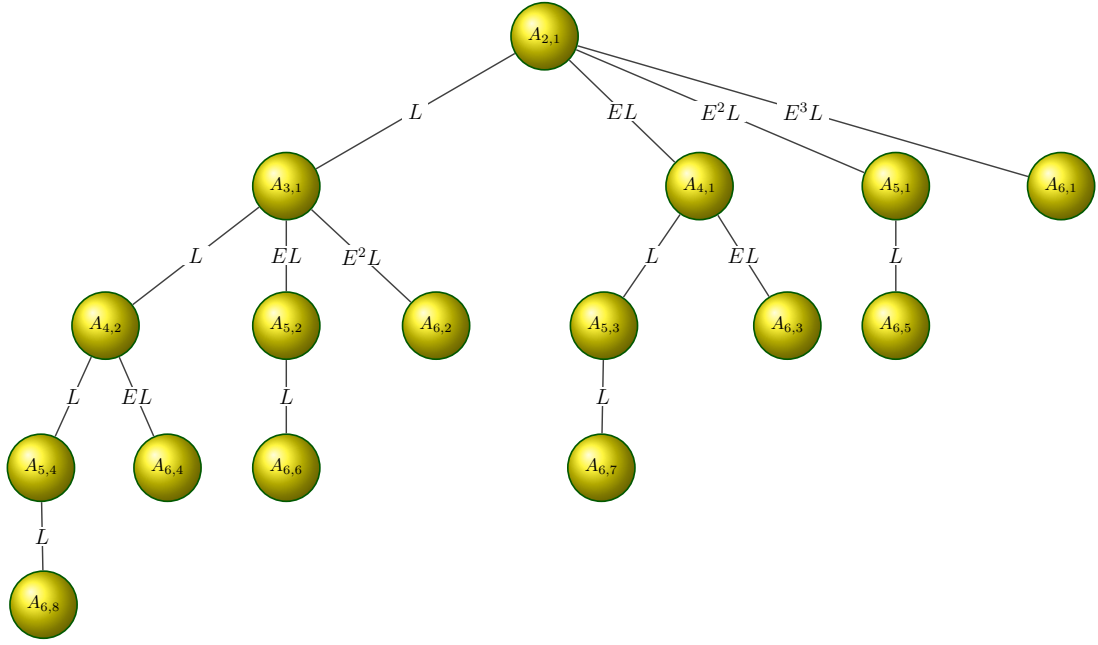

\subsection{Initial associated families}
\label{subsec:initial-associated-families}

By Lemma~\ref{lem:parent-endpoint}, every parent chain ends at an element
of
$$
I_2=\{4,6,8,10\}.
$$
We first define the associated polynomial families for these four values.

Let
$$
\Delta(Y_0)=Y_0^{q-1},
\qquad
\gamma(Y_0)=1-\Delta(Y_0),
$$
and
$$
\alpha(Y_0,Y_1)
=
\Delta(Y_0)Y_1(Y_1+Y_0).
$$
For $y_0\in\mathbb F_q$,
$$
\Delta(y_0)
=
\begin{cases}
0,&y_0=0,\\
1,&y_0\neq0,
\end{cases}
\qquad
\gamma(y_0)
=
\begin{cases}
1,&y_0=0,\\
0,&y_0\neq0.
\end{cases}
$$

Define
\begin{equation}
\Omega(Y_0,Y_1)
=
\Delta(Y_0)+\Delta(Y_1)-\Delta(Y_0)\Delta(Y_1).
\label{eq:Omega}
\end{equation}
and
\begin{equation}
\begin{aligned}
\ell_0(Y_0,Y_1)&=1,\qquad
\ell_1(Y_0,Y_1)&=Y_1, \qquad 
\ell_2(Y_0,Y_1)&=Y_0, \qquad
\ell_3(Y_0,Y_1)&=Y_0+Y_1.
\end{aligned}
\label{eq:ell-j}
\end{equation}
For $m\in I_2=\{4,6,8,10\}$, set
\begin{equation}
\sigma_m(Y_0,Y_1)
=
\Omega(Y_0,Y_1)
\prod_{j=0}^{(m-4)/2}\ell_j(Y_0,Y_1),
\end{equation}

and define 
\begin{equation}
    \mathcal F^{(m)}
=
\bigl(Q(Y_0),\sigma_m(Y_0,Y_1)\bigr).
\label{eq:initial-Fm}
\end{equation}

\begin{prop}\label{prop:initial-even-systems}
Let $m\in\{4,6,8,10\}$, and let $\mathcal F^{(m)}$ be defined by
\eqref{eq:initial-Fm}. If $\mathcal G^{(m)}$ is a defining family
whose associated family is $\mathcal F^{(m)}$, then
$$
\#S_{\mathcal G^{(m)}}(\mathbb F_q)
=
P_m(q).
$$
\end{prop}

\begin{proof}
Since $Q(y_0)=0$ for every $y_0\in\mathbb F_q$,
$
N_0(\mathcal G^{(m)})
$
is the number of zeros of
$
\sigma_m(Y_0,Y_1)
$
in $\mathbb F_q^2$. Moreover,
$
\Omega(y_0,y_1)=0
$
if and only if
$
(y_0,y_1)=(0,0).
$

We consider the four values of $m$.

\begin{enumerate}
\item For $m=4$,
$
\sigma_4(Y_0,Y_1)=\Omega(Y_0,Y_1).
$
Hence its only zero is $(0,0)$, and therefore
$
N_0(\mathcal G^{(4)})=1.
$

\item For $m=6$,
$
\sigma_6(Y_0,Y_1)
=
\Omega(Y_0,Y_1)Y_1.
$
Thus $\sigma_6=0$ if and only if
$
\Omega(Y_0,Y_1)=0
$
or
$
Y_1=0.
$
Since the only zero of $\Omega$ is $(0,0)$, which already lies on
$Y_1=0$, the zero set of $\sigma_6$ is the line $Y_1=0$. Hence
$
N_0(\mathcal G^{(6)})=q.
$

\item For $m=8$,
$
\sigma_8(Y_0,Y_1)
=
\Omega(Y_0,Y_1)Y_0Y_1.
$
Its zero set is the union of the two lines
$$
Y_0=0
\qquad\text{and}\qquad
Y_1=0,
$$
since the zero of $\Omega$ is their common point $(0,0)$. Each line
contains $q$ points, and they intersect only at $(0,0)$. Therefore
$
N_0(\mathcal G^{(8)})
=
2q-1.
$

\item For $m=10$,
$
\sigma_{10}(Y_0,Y_1)
=
\Omega(Y_0,Y_1)Y_0Y_1(Y_0+Y_1).
$
Its zero set is the union of the three lines
$$
Y_0=0,
\qquad
Y_1=0,
\qquad
Y_0+Y_1=0.
$$
These lines are distinct and meet pairwise only at $(0,0)$. Hence
their union has
$
1+3(q-1)=3q-2
$
points. Thus
$
N_0(\mathcal G^{(10)})=3q-2.
$
\end{enumerate}

Hence, for every $m\in\{4,6,8,10\}$,
$$
N_0(\mathcal G^{(m)})
=
1+\frac{m-4}{2}(q-1).
$$
By Theorem~\ref{thm q(q-1)+qN_0},
$$
\#S_{\mathcal G^{(m)}}(\mathbb F_q)
=
q(q-1)
+
q\left(
1+\frac{m-4}{2}(q-1)
\right)
=
P_m(q).
$$
\end{proof}

These four families form the initial associated cases of the recursive construction.
We next define the operations used to construct $\mathcal F^{(m)}$
from $\mathcal F^{(\Par(m))}$.

\subsection{Recursive operators}\label{subsec:recursive-operations}

Let $m\in I_n$ with $n\geq3$, and suppose that $\Par(m)\in I_r$.
By Lemma~\ref{lem:polygonal-parent-difference},
$$
P_m(q)-P_{\Par(m)}(q)
=
2^{n-1}q(q-1).
$$
Hence Theorem~\ref{thm q(q-1)+qN_0} shows that passing from
$P_{\Par(m)}(q)$ to $P_m(q)$ requires an increase of
$
2^{n-1}(q-1)
$
in $N_0$.

We define two operators on associated families to obtain this increase.

Let
$
\mathcal G=(G_1,\ldots,G_s)
$
be a defining family as in \eqref{eq:defining-polynomials}, and let
$
\mathcal F=(F_1,\ldots,F_s)
$
be its associated family as in \eqref{eq:associated-family}, where
$$
F_i\in\mathbb F_q[Y_0,\ldots,Y_{i-1}],
\qquad
1\leq i\leq s.
$$

For $1\leq i\leq s$, let
$$
\widehat F_i(Y_1,\ldots,Y_i)
=
F_i(Y_1,\ldots,Y_i),
$$
where $\widehat F_i$ is obtained from $F_i$ by replacing $Y_j$ by
$Y_{j+1}$ for $0\leq j\leq i-1$.

\begin{defi}\label{def:L-operation}
Define the operator $L$ by
$$
L(\mathcal F)
=
\left(
Y_0^q-Y_0,\,
\alpha(Y_0,Y_1)+\gamma(Y_0)\widehat F_1(Y_1),\,
\ldots,\,
\alpha(Y_0,Y_s)+\gamma(Y_0)\widehat F_s(Y_1,\ldots,Y_s)
\right).
$$
Then $L(\mathcal F)$ consists of $s+1$ polynomials in
$\mathbb F_q[Y_0,\ldots,Y_s]$.
\end{defi}

\begin{ex}\label{ex:L-F6}
For
$
\mathcal F^{(6)}=(Y_0^q-Y_0,Y_1),
$
we have
$$
\widehat F_1(Y_1)=Y_1^q-Y_1,
\qquad
\widehat F_2(Y_1,Y_2)=Y_2.
$$
Hence
$$
L(\mathcal F^{(6)})
=
\left(
Y_0^q-Y_0,\,
\alpha(Y_0,Y_1)+\gamma(Y_0)(Y_1^q-Y_1),\,
\alpha(Y_0,Y_2)+\gamma(Y_0)Y_2
\right).
$$
\end{ex}

The next lemma gives the change in $N_0$ under the operator $L$.

\begin{lemma}\label{lem:L-N0}
Let $\mathcal G$ be a defining family with associated family
$
\mathcal F=(F_1,\ldots,F_s)
$
as in \eqref{eq:associated-family}. Let $\mathcal G_L$ be a defining
family whose associated family is $L(\mathcal F)$. Then
$$
N_0(\mathcal G_L)
=
N_0(\mathcal G)+2^s(q-1).
$$
\end{lemma}
\begin{proof}
By definition of $L(\mathcal F)$, the solutions counted by
$N_0(\mathcal G_L)$ are the tuples
$
(y_0,y_1,\ldots,y_s)\in\mathbb F_q^{s+1}
$
satisfying
$$
y_0^q-y_0=0
$$
and
$$
\alpha(y_0,y_i)
+
\gamma(y_0)\widehat F_i(y_1,\ldots,y_i)
=
0,
\qquad
1\leq i\leq s.
$$
Since $y_0^q-y_0=0$ for every $y_0\in\mathbb F_q$, we separate the
solutions according as $y_0=0$ or $y_0\neq0$.

If $y_0=0$, then
$
\gamma(y_0)=1
$
and
$
\alpha(y_0,y_i)=0.
$
Hence the equations reduce to
$$
\widehat F_i(y_1,\ldots,y_i)=0,
\qquad
1\leq i\leq s.
$$
After relabelling $y_1,\ldots,y_s$ as $y_0,\ldots,y_{s-1}$, these are
the equations defining $N_0(\mathcal G)$. Thus the contribution from
$y_0=0$ is $N_0(\mathcal G)$.

If $y_0\neq0$, then
$
\gamma(y_0)=0
$
and
$$
\alpha(y_0,y_i)
=
y_i(y_i+y_0).
$$
Thus
$$
y_i\in\{0,-y_0\},
\qquad
1\leq i\leq s.
$$
For each of the $q-1$ choices of $y_0\in\mathbb F_q^\times$, there are
$2^s$ choices for $(y_1,\ldots,y_s)$. Hence the contribution from
$y_0\neq0$ is $2^s(q-1)$.

Therefore
$$
N_0(\mathcal G_L)
=
N_0(\mathcal G)+2^s(q-1).
$$
\end{proof}

Let $m\in I_n$ with $\Par(m)\in I_r$. If $L$ is applied to a family
of $r$ polynomials, Lemma~\ref{lem:L-N0} shows that $N_0$ increases by
$
2^r(q-1).
$
If $r=n-1$, this is the required increase.

If $r<n-1$, applying $L$ again is not suitable. A second
application of $L$ gives a total increase
$$
2^r(q-1)+2^{r+1}(q-1)
$$
in $N_0$, instead of the required
$
2^{n-1}(q-1).
$
We therefore introduce a second operator that leaves the contribution to
$N_0$ from $y_0=0$ unchanged and doubles the contribution from
$y_0\neq0$.

For $i\geq1$, set
\begin{equation}
\beta(Y_0,Y_i)
=
\alpha(Y_0,Y_i)+\gamma(Y_0)Y_i.
\label{eq:beta}
\end{equation}

\begin{defi}\label{def:extension}
For
$
\mathcal F=(F_1,\ldots,F_s),
$
define the operator $E$ by 
$$
E(\mathcal F)
=
\left(
F_1,\ldots,F_s,\beta(Y_0,Y_s)
\right).
$$
Set $E^0(\mathcal F)=\mathcal F$, and for $\ell\geq0$ define
$$
E^{\ell+1}(\mathcal F)
=
E\bigl(E^\ell(\mathcal F)\bigr).
$$
Thus, for $\ell\geq1$,
$$
E^\ell(\mathcal F)
=
\left(
F_1,\ldots,F_s,\,
\beta(Y_0,Y_s),\,
\beta(Y_0,Y_{s+1}),\,
\ldots,\,
\beta(Y_0,Y_{s+\ell-1})
\right).
$$
Then $E^\ell(\mathcal F)$ consists of $s+\ell$ polynomials in
$\mathbb F_q[Y_0,\ldots,Y_{s+\ell-1}]$.
\end{defi}

\begin{ex}\label{ex:E-F6}
Let
$
L(\mathcal F^{(6)})=(K_1,K_2,K_3)
$
be the family in Example~\ref{ex:L-F6}. Then
$$
E\left(L(\mathcal F^{(6)})\right)
=
\left(
K_1,K_2,K_3,\beta(Y_0,Y_3)
\right),
$$
and
$$
E^2\left(L(\mathcal F^{(6)})\right)
=
\left(
K_1,K_2,K_3,
\beta(Y_0,Y_3),
\beta(Y_0,Y_4)
\right).
$$
\end{ex}

The next lemma gives $N_0$ for $E^\ell(L(\mathcal F))$.
\begin{lemma}\label{lem:LE-N0}
Let $\mathcal G$ be a defining family with associated family
$
\mathcal F=(F_1,\ldots,F_s)
$
as in \eqref{eq:associated-family}. For $\ell\geq0$, let
$\mathcal G_\ell$ be a defining family whose associated family is
$
E^\ell(L(\mathcal F)).
$
Then
$
E^\ell(L(\mathcal F))
$
consists of $s+\ell+1$ polynomials in
$
\mathbb F_q[Y_0,\ldots,Y_{s+\ell}],
$
and
$$
N_0(\mathcal G_\ell)
=
N_0(\mathcal G)+2^{s+\ell}(q-1).
$$
\end{lemma}

\begin{proof}
By Lemma~\ref{lem:L-N0}, after applying $L$, the contribution from
$y_0=0$ is $N_0(\mathcal G)$, while the contribution from $y_0\neq0$
is $2^s(q-1)$.

For $E^\ell(L(\mathcal F))$, the additional equations are
$$
\beta(Y_0,Y_{s+1})=0,\,
\beta(Y_0,Y_{s+2})=0,\,
\ldots,\,
\beta(Y_0,Y_{s+\ell})=0.
$$
If $y_0=0$, then
$$
\beta(0,y_i)=y_i,
\qquad
s+1\leq i\leq s+\ell,
$$
so
$$
y_{s+1}=\cdots=y_{s+\ell}=0.
$$
Hence the contribution from $y_0=0$ remains $N_0(\mathcal G)$.

If $y_0\neq0$, then
$$
\beta(y_0,y_i)
=
y_i(y_i+y_0),
\qquad
s+1\leq i\leq s+\ell.
$$
Thus
$$
y_i\in\{0,-y_0\},
\qquad
s+1\leq i\leq s+\ell.
$$
The $\ell$ new variables give $2^\ell$ choices. Hence the contribution
from $y_0\neq0$ is
$$
2^s(q-1)2^\ell
=
2^{s+\ell}(q-1).
$$
Therefore
$$
N_0(\mathcal G_\ell)
=
N_0(\mathcal G)+2^{s+\ell}(q-1).
$$

Further, since $L(\mathcal F)$ has $s+1$ polynomials and $E^\ell$ adds $\ell$
polynomials, $E^\ell(L(\mathcal F))$ has $s+\ell+1$ polynomials.
\end{proof}

Let $m\in I_n$ and suppose that $\Par(m)\in I_r$. If $\mathcal F$
consists of $r$ polynomials, then Lemma~\ref{lem:LE-N0} shows that
$
E^\ell(L(\mathcal F))
$
has $r+\ell+1$ polynomials and increases $N_0$ by
$
2^{r+\ell}(q-1).
$
To obtain the required increase $2^{n-1}(q-1)$, we need
$$
2^{r+\ell}(q-1)
=
2^{n-1}(q-1),
$$
and hence
$$
\ell=n-r-1.
$$

If $m\in A_{n,j}$ and $\Par(m)\in A_{r,s}$, then the corresponding recursive step is given by the operator $E^{\,n-r-1}\circ L$, as indicated in Figure~\ref{fig:block-dependency-B4}.

\begin{defi}\label{def:extension-number}
For every even integer $m\geq12$, define
$$
\nu(m)
=
n(m)-n(\Par(m))-1.
$$
\end{defi}

If $m\in I_n$ and $\Par(m)\in I_r$, then
Corollary~\ref{cor:parent-bijection} gives $r<n$. Hence
$$
\nu(m)=n-r-1\geq0.
$$

\begin{defi}\label{def:recursive-even-family}
For every even integer $m\geq12$, define recursively
$$
\mathcal F^{(m)}
=
E^{\nu(m)}
\left(
L\bigl(\mathcal F^{(\Par(m))}\bigr)
\right).
$$
\end{defi}

By Lemma~\ref{lem:parent-termination}, repeated application of $\Par$
on $m$ reaches an element of $I_2$. Hence Definition~\ref{def:recursive-even-family},
together with the initial families in \eqref{eq:initial-Fm}, defines
$\mathcal F^{(m)}$ for every even integer $m\geq4$. 

\begin{lemma}\label{lem:recursive-family-size}
Let $m\in I_n$. Then the associated family $\mathcal F^{(m)}$ defined
by \eqref{eq:initial-Fm} when $n=2$, and by
Definition~\ref{def:recursive-even-family} when $n\geq3$, consists of
$n=n(m)$ polynomials.
\end{lemma}

\begin{proof}
We use induction on $n$.

If $n=2$, then
$
m\in I_2=\{4,6,8,10\},
$
and each family in \eqref{eq:initial-Fm} consists of two polynomials.

Let $n\geq3$ and suppose that
$
\Par(m)\in I_r.
$
Since $r<n$, the induction hypothesis shows that
$\mathcal F^{(\Par(m))}$ consists of $r$ polynomials. Since
$
\nu(m)=n-r-1,
$
Definitions~\ref{def:L-operation}, \ref{def:extension}, and
\ref{def:recursive-even-family} show that
$\mathcal F^{(m)}$ consists of
$$
r+1+\nu(m)
=
r+1+(n-r-1)
=
n
$$
polynomials.
\end{proof}
Let $m\in I_n$ with $\Par(m)\in I_r$. Let
$\mathcal G^{(m)}$ and $\mathcal G^{(\Par(m))}$ be defining families
whose associated families, as in \eqref{eq:associated-family}, are
$\mathcal F^{(m)}$ and $\mathcal F^{(\Par(m))}$, respectively.
By
Lemma~\ref{lem:recursive-family-size},
$\mathcal F^{(\Par(m))}$ consists of $r$ polynomials. Since
$
\nu(m)=n-r-1,
$
Lemma~\ref{lem:LE-N0} gives
\begin{equation}\label{eq:recursive-N0}
N_0(\mathcal G^{(m)})
=
N_0(\mathcal G^{(\Par(m))})
+
2^{n-1}(q-1).
\end{equation}

If
$$
m=m_t\longmapsto m_{t-1}\longmapsto\cdots
\longmapsto m_1\longmapsto m_0
$$
is the parent chain of $m$, then repeated application of
Definition~\ref{def:recursive-even-family} gives
$$
\mathcal F^{(m)}
=
\left(E^{\nu(m_t)}\circ L\right)
\circ
\left(E^{\nu(m_{t-1})}\circ L\right)
\circ\cdots\circ
\left(E^{\nu(m_1)}\circ L\right)
\left(\mathcal F^{(m_0)}\right).
$$

Figure~\ref{fig:recursive-step} summarizes the passage from $\mathcal F^{(\Par(m))}$ to $\mathcal F^{(m)}$, together with the corresponding change in $N_0$ and in the number of polynomials.

\begin{figure}[htbp]
\centering
\small
\setlength{\fboxsep}{8pt}
\fbox{%
\begin{minipage}{0.92\textwidth}
\centering
\renewcommand{\arraystretch}{1.8}
$
\begin{array}{@{}r@{\qquad}c@{\quad}c@{\quad}c@{\quad}c@{\quad}c@{}}
\text{\rm even value}
&
\Par(m)\in I_r
&
\multicolumn{3}{c}{
\xrightarrow{\qquad m-\Par(m)=2^n\qquad}
}
&
m\in I_n
\\
\text{\rm associated family}
&
\mathcal F^{(\Par(m))}
&
\xrightarrow{\ L\ }
&
L\left(\mathcal F^{(\Par(m))}\right)
&
\xrightarrow{\ E^{\nu(m)}\ }
&
\mathcal F^{(m)}
\\
\text{\rm increase in }N_0
&
0
&
\xrightarrow{\ L\ }
&
2^r(q-1)
&
\xrightarrow{\ \times 2^{\nu(m)}\ }
&
2^{n-1}(q-1)
\\

\text{\rm number of polynomials}
&
r
&
\xrightarrow{\ +1\ }
&
r+1
&
\xrightarrow{\ +\nu(m)\ }
&
n
\end{array}
$
\end{minipage}%
}
\caption{The recursive step from $\Par(m)$ to $m$.}
\label{fig:recursive-step}
\end{figure}

\subsection{Realization of even polygonal numbers}

We now prove the point-count formula for the recursively constructed families.

\begin{theorem}\label{thm:even-polygonal-realization}
For every even integer $m\geq4$ and every prime power $q$, there exists
a defining family $\mathcal G^{(m)}$ such that
$$
\#S_{\mathcal G^{(m)}}(\mathbb F_q)
=
P_m(q)
=
\frac{(m-2)q^2-(m-4)q}{2}.
$$
Consequently,
$$
\left(
\#S_{\mathcal G^{(m)}}(\mathbb F_q)
\right)_{q\in\mathcal Q}
=
\left(
P_m(q)
\right)_{q\in\mathcal Q},
$$
the prime-power subsequence of the $m$-gonal numbers.
\end{theorem}
\begin{proof}
By \eqref{eq:Gi-FR}, for each associated family $\mathcal F^{(m)}$
there exists a defining family $\mathcal G^{(m)}$ having
$\mathcal F^{(m)}$ as its associated family.

If $m\in I_2=\{4,6,8,10\}$, the result follows from
Proposition~\ref{prop:initial-even-systems}.

Now let $m\in I_n$ with $n\geq3$, and let
$$
m=m_t\longmapsto m_{t-1}\longmapsto\cdots
\longmapsto m_1\longmapsto m_0
$$
be its parent chain. We proceed by induction along this parent chain. For each $0\leq i\leq t$, let
$\mathcal G^{(m_i)}$ be a defining family whose associated family is
$\mathcal F^{(m_i)}$. We successively prove that
$$
\#S_{\mathcal G^{(m_i)}}(\mathbb F_q)
=
P_{m_i}(q),
\qquad
0\leq i\leq t.
$$

Since $m_0\in I_2=\{4,6,8,10\}$, the case $i=0$ follows from
Proposition~\ref{prop:initial-even-systems}.

Let $1\leq i\leq t$, and suppose that
$$
\#S_{\mathcal G^{(m_{i-1})}}(\mathbb F_q)
=
P_{m_{i-1}}(q).
$$
Write
$
m_i\in I_s
$
and
$
m_{i-1}=\Par(m_i)\in I_r.
$
By \eqref{eq:recursive-N0},
$$
N_0(\mathcal G^{(m_i)})
=
N_0(\mathcal G^{(m_{i-1})})
+
2^{s-1}(q-1).
$$
Therefore, by Theorem~\ref{thm q(q-1)+qN_0},
$$
\#S_{\mathcal G^{(m_i)}}(\mathbb F_q)
-
\#S_{\mathcal G^{(m_{i-1})}}(\mathbb F_q)
=
2^{s-1}q(q-1).
$$
By Lemma~\ref{lem:polygonal-parent-difference},
$$
P_{m_i}(q)-P_{m_{i-1}}(q)
=
2^{s-1}q(q-1).
$$
Hence
$$
\#S_{\mathcal G^{(m_i)}}(\mathbb F_q)
=
\#S_{\mathcal G^{(m_{i-1})}}(\mathbb F_q)
+
2^{s-1}q(q-1).
$$
Using the induction hypothesis
$
\#S_{\mathcal G^{(m_{i-1})}}(\mathbb F_q)
=
P_{m_{i-1}}(q),
$
we get
$$
\#S_{\mathcal G^{(m_i)}}(\mathbb F_q)
=
P_{m_{i-1}}(q)
+
2^{s-1}q(q-1)
=
P_{m_i}(q).
$$
Taking $i=t$ gives
$$
\#S_{\mathcal G^{(m)}}(\mathbb F_q)
=
P_m(q)
=
\frac{(m-2)q^2-(m-4)q}{2}.
$$
The equality of the prime-power subsequences follows.
\end{proof}

\begin{ex}\label{ex:recursive-358}
Consider the parent chain from Example~\ref{ex:parent-chain-358},
$$
358\longmapsto102\longmapsto38\longmapsto6.
$$
By Proposition~\ref{prop:level-of-m},
$$
\begin{aligned}
n(358)&=\left\lfloor\log_2(358-4)\right\rfloor=8,\\
n(102)&=\left\lfloor\log_2(102-4)\right\rfloor=6,\\
n(38)&=\left\lfloor\log_2(38-4)\right\rfloor=5.
\end{aligned}
$$
Since $6\in I_2$, Definition~\ref{def:In} gives $n(6)=2$.
Hence Definition~\ref{def:extension-number} gives
$$
\nu(38)=2,
\qquad
\nu(102)=0,
\qquad
\nu(358)=1.
$$

By Definition~\ref{def:recursive-even-family},
$$
\mathcal F^{(38)}
=
E^2\left(L(\mathcal F^{(6)})\right),
$$
$$
\mathcal F^{(102)}
=
L(\mathcal F^{(38)}),
$$
and
$$
\mathcal F^{(358)}
=
E\left(L(\mathcal F^{(102)})\right).
$$
Thus
$$
\mathcal F^{(6)}
\xrightarrow{\,E^2\circ L\,}
\mathcal F^{(38)}
\xrightarrow{\,L\,}
\mathcal F^{(102)}
\xrightarrow{\,E\circ L\,}
\mathcal F^{(358)},
$$
or equivalently,
$$
\mathcal F^{(358)}
=
(E\circ L)\circ L\circ(E^2\circ L)
\left(\mathcal F^{(6)}\right).
$$

Writing
$$
\mathcal F^{(358)}
=
\left(
F_1^{(358)},\ldots,F_8^{(358)}
\right),
$$
successive applications of Definitions~\ref{def:L-operation}
and~\ref{def:extension} give
\begin{align*}
F_1^{(358)}
&=Y_0^q-Y_0,\\[1mm]
F_2^{(358)}
&=Y_0^{q-1}Y_1(Y_1+Y_0)
 +(1-Y_0^{q-1})(Y_1^q-Y_1),\\[1mm]
F_3^{(358)}
&=Y_0^{q-1}Y_2(Y_2+Y_0)\\
&\quad +(1-Y_0^{q-1})
\left[
Y_1^{q-1}Y_2(Y_2+Y_1)
 +(1-Y_1^{q-1})(Y_2^q-Y_2)
\right],\\[1mm]
F_4^{(358)}
&=Y_0^{q-1}Y_3(Y_3+Y_0)\\
&\quad +(1-Y_0^{q-1})
\left[
Y_1^{q-1}Y_3(Y_3+Y_1)
 +(1-Y_1^{q-1})
\left(
Y_2^{q-1}Y_3(Y_3+Y_2)
 +(1-Y_2^{q-1})(Y_3^q-Y_3)
\right)
\right],\\[1mm]
F_5^{(358)}
&=Y_0^{q-1}Y_4(Y_4+Y_0)\\
&\quad +(1-Y_0^{q-1})
\left[
Y_1^{q-1}Y_4(Y_4+Y_1)
 +(1-Y_1^{q-1})
\left(
Y_2^{q-1}Y_4(Y_4+Y_2)
 +(1-Y_2^{q-1})Y_4
\right)
\right],\\[1mm]
F_6^{(358)}
&=Y_0^{q-1}Y_5(Y_5+Y_0)\\
&\quad +(1-Y_0^{q-1})
\left[
Y_1^{q-1}Y_5(Y_5+Y_1)
 +(1-Y_1^{q-1})
\left(
Y_2^{q-1}Y_5(Y_5+Y_2)
 +(1-Y_2^{q-1})Y_5
\right)
\right],\\[1mm]
F_7^{(358)}
&=Y_0^{q-1}Y_6(Y_6+Y_0)\\
&\quad +(1-Y_0^{q-1})
\left[
Y_1^{q-1}Y_6(Y_6+Y_1)
 +(1-Y_1^{q-1})
\left(
Y_2^{q-1}Y_6(Y_6+Y_2)
 +(1-Y_2^{q-1})Y_6
\right)
\right],\\[1mm]
F_8^{(358)}
&=Y_0^{q-1}Y_7(Y_7+Y_0)
 +(1-Y_0^{q-1})Y_7.
\end{align*}

Thus $\mathcal F^{(358)}$ consists of eight polynomials. By
Theorem~\ref{thm:even-polygonal-realization}, every defining family
$\mathcal G^{(358)}$ whose associated family is
$\mathcal F^{(358)}$ satisfies
$$
\#S_{\mathcal G^{(358)}}(\mathbb F_q)
=
P_{358}(q).
$$
\end{ex}

\section{Polynomial families realizing odd polygonal numbers}
\label{sec:odd-polygonal-surfaces}

Let $m\geq3$ be odd. By Corollary~\ref{cor:polygonal-bound}, no family
$\mathcal G$ of polynomials satisfying
\eqref{eq:defining-polynomials} can satisfy
$$
\#S_{\mathcal G}(\mathbb F_q)=P_m(q)
$$
for every prime power $q$. Since $m+1$ is even, Theorem~\ref{thm:even-polygonal-realization}
provides a defining family $\mathcal G^{(m+1)}$ such that
$$
\#S_{\mathcal G^{(m+1)}}(\mathbb F_q)
=
P_{m+1}(q).
$$
By Corollary~\ref{cor:x-nonzero-fiber},
$S_{\mathcal G^{(m+1)}}(\mathbb F_q)$ has $q(q-1)$ points with
$x\neq0$.
Since
$$
P_{m+1}(q)-P_m(q)
=
\frac{q(q-1)}{2},
$$
it is enough to retain all points with $x=0$ and exactly half of the
points with $x\neq0$. This motivates the auxiliary polynomial introduced
below.

\subsection{The auxiliary polynomial}
\label{subsec:odd-auxiliary-polynomial}

Let $q=p^r$, where $p$ is prime. For $a\in\mathbb F_q$, the absolute
trace of $a$ is
$$
\operatorname{Tr}_{\mathbb F_q/\mathbb F_p}(a)
=
a+a^p+a^{p^2}+\cdots+a^{p^{r-1}}.
$$

Set
$$
\Lambda_q
=
\frac{q(q-1)}{2},
$$
and define
$$
\Theta(X,Y_0)
=
2^{q-1}X\left(X^{\Lambda_q}-1\right)
+
\left(1-2^{q-1}\right)
X^{q-1}
\operatorname{Tr}_{\mathbb F_q/\mathbb F_p}(Y_0).
$$

For a defining family
$
\mathcal G=(G_1,\ldots,G_n)
$
as in \eqref{eq:defining-polynomials}, with
$S_{\mathcal G}(\mathbb F_q)$ as in \eqref{eq:SG}, define
$$
S_{\mathcal G}^{\Theta}(\mathbb F_q)
=
\left\{
(x,y_0,\ldots,y_n)\in S_{\mathcal G}(\mathbb F_q):
\Theta(x,y_0)=0
\right\}.
$$

We recall two standard finite-field counting results.

\begin{lemma}[\cite{IrelandRosen1990}, Proposition 7.1.2]
\label{lem:binomial-solvability}
Let $a\in\mathbb F_q^\times$, let $N\geq1$, and set
$
d=\gcd(N,q-1).
$
Then the equation
$
X^N=a
$
has a solution in $\mathbb F_q$ if and only if
$
a^{(q-1)/d}=1.
$
If it has a solution, then it has exactly $d$ solutions in
$\mathbb F_q$.
\end{lemma}

\begin{lemma}[\cite{mullen2013handbook}, Theorem 2.1.83(7)]
\label{lem:trace-fiber}
Let $F/K$ be a finite extension of finite fields with
$
[F:K]=n
$
and
$
|K|=N.
$
Then, for every $\alpha\in K$,
$$
\#\left\{
\beta\in F:
\operatorname{Tr}_{F/K}(\beta)=\alpha
\right\}
=
N^{n-1}.
$$
\end{lemma}

\begin{lemma}\label{lem:theta-half-count}
For every $y_0\in\mathbb F_q$,
$
\Theta(0,y_0)=0.
$
Moreover,
$$
\#\left\{
(x,y_0)\in\mathbb F_q^\times\times\mathbb F_q:
\Theta(x,y_0)=0
\right\}
=
\frac{q(q-1)}{2}.
$$
\end{lemma}

\begin{proof}
For every $y_0\in\mathbb F_q$, we have
$
\Theta(0,y_0)=0.
$

We consider two cases.

\medskip
\noindent
Case 1: Suppose that $p$ is odd. Then $2\neq0$ in $\mathbb F_q$, and hence
$
2^{q-1}=1.
$
Therefore
$$
\Theta(X,Y_0)
=
X\left(X^{\Lambda_q}-1\right).
$$
For $x\in\mathbb F_q^\times$, the equation
$
\Theta(x,y_0)=0
$
is equivalent to
$
x^{\Lambda_q}=1.
$

Since $p$ is odd, $q$ is odd, and hence $(q-1)/2$ is an integer. Moreover,
$$
\Lambda_q
=
\frac{q(q-1)}{2}
=
\frac{q-1}{2}
+
\frac{q-1}{2}(q-1).
$$
Thus, using $x^{q-1}=1$,
$$
x^{\Lambda_q}
=
x^{(q-1)/2}
\left(x^{q-1}\right)^{(q-1)/2}
=
x^{(q-1)/2}.
$$
Hence
$$
x^{\Lambda_q}=1
$$
is equivalent to
$$
x^{(q-1)/2}=1.
$$
By Lemma~\ref{lem:binomial-solvability}, applied with
$
a=1
$
and
$
N=(q-1)/2,
$
this equation has
$$
\gcd\left(\frac{q-1}{2},q-1\right)
=
\frac{q-1}{2}
$$
solutions in $\mathbb F_q^\times$. Since $y_0$ is arbitrary, the number
of pairs $(x,y_0)$ is
$$
q\cdot\frac{q-1}{2}
=
\frac{q(q-1)}{2}.
$$

\medskip
\noindent
Case 2: Suppose that $p=2$. Then $2=0$ in $\mathbb F_q$, and hence
$
2^{q-1}=0.
$
Therefore
$$
\Theta(X,Y_0)
=
X^{q-1}
\operatorname{Tr}_{\mathbb F_q/\mathbb F_2}(Y_0).
$$
For $x\neq0$, we have $x^{q-1}=1$. Hence
$
\Theta(x,y_0)=0
$
is equivalent to
$$
\operatorname{Tr}_{\mathbb F_q/\mathbb F_2}(y_0)=0.
$$
Since $q=2^r$, Lemma~\ref{lem:trace-fiber}, applied to
$
\mathbb F_q/\mathbb F_2,
$
gives
$$
\#\left\{
y_0\in\mathbb F_q:
\operatorname{Tr}_{\mathbb F_q/\mathbb F_2}(y_0)=0
\right\}
=
2^{r-1}
=
\frac{q}{2}.
$$
Thus there are $q-1$ choices for $x$ and $q/2$ choices for $y_0$.
Therefore the number of pairs $(x,y_0)$ is
$$
(q-1)\frac{q}{2}
=
\frac{q(q-1)}{2}.
$$
\end{proof}

\subsection{Construction for odd values of $m$}
\label{subsec:odd-polygonal-construction}

\begin{defi}\label{def:odd-family}
Let $m\geq3$ be odd, and let
$$
\mathcal G^{(m+1)}
=
\left(G_1^{(m+1)},\ldots,
G_{n(m+1)}^{(m+1)}\right)
$$
be a defining family as in
Theorem~\ref{thm:even-polygonal-realization}.
Define
$$
S_{\mathcal G^{(m+1)}}^\Theta(\mathbb F_q)
=
\left\{
(x,y_0,\ldots,y_{n(m+1)})
\in S_{\mathcal G^{(m+1)}}(\mathbb F_q):
\Theta(x,y_0)=0
\right\}.
$$
\end{defi}

\begin{theorem}\label{thm:odd-polygonal-realization}
For every odd integer $m\geq3$ and every prime power $q$, there exists
a defining family $\mathcal G^{(m+1)}$ such that
$$
\#S_{\mathcal G^{(m+1)}}^\Theta(\mathbb F_q)
=
P_m(q)
=
\frac{(m-2)q^2-(m-4)q}{2}.
$$
Consequently,
$$
\left(
\#S_{\mathcal G^{(m+1)}}^\Theta(\mathbb F_q)
\right)_{q\in\mathcal Q}
=
\left(
P_m(q)
\right)_{q\in\mathcal Q},
$$
the prime-power subsequence of the $m$-gonal numbers.
\end{theorem}

\begin{proof}
Let $m\geq3$ be odd, and let $q$ be a prime power.
Since $m+1$ is even,
Theorem~\ref{thm:even-polygonal-realization} gives
a defining family $\mathcal G^{(m+1)}$ such that
$$
\#S_{\mathcal G^{(m+1)}}(\mathbb F_q)
=
P_{m+1}(q)
=
\frac{(m-1)q^2-(m-3)q}{2}.
$$

By Lemma~\ref{lem:theta-half-count},
$\Theta(0,y_0)=0$
for every $y_0\in\mathbb F_q$. Hence all points with $x=0$
remain in $S_{\mathcal G^{(m+1)}}^\Theta(\mathbb F_q)$.

By Corollary~\ref{cor:x-nonzero-fiber},
the points of $S_{\mathcal G^{(m+1)}}(\mathbb F_q)$
with $x\neq0$ are in bijection with
$\mathbb F_q^\times\times\mathbb F_q$.
By Lemma~\ref{lem:theta-half-count}, exactly
$q(q-1)/2$ of these points satisfy $\Theta(x,y_0)=0$.
Consequently,
$$
\#S_{\mathcal G^{(m+1)}}(\mathbb F_q)
-
\#S_{\mathcal G^{(m+1)}}^\Theta(\mathbb F_q)
=
\frac{q(q-1)}{2}.
$$
Therefore
$$
\begin{aligned}
\#S_{\mathcal G^{(m+1)}}^\Theta(\mathbb F_q)
&=
P_{m+1}(q)-\frac{q(q-1)}{2}\\
&=
\frac{(m-2)q^2-(m-4)q}{2}\\
&=
P_m(q).
\end{aligned}
$$
The equality of the prime-power subsequences follows.
\end{proof}

\begin{ex}
For $m=357$, Example~\ref{ex:recursive-358} gives the defining family
$\mathcal G^{(358)}$. By Definition~\ref{def:odd-family},
$$
\mathcal G^{(357)}
=
\left(
G_1^{(358)},\ldots,G_8^{(358)},\Theta
\right).
$$
Hence Theorem~\ref{thm:odd-polygonal-realization} gives
$$
\#S_{\mathcal G^{(357)}}(\mathbb F_q)
=
P_{357}(q).
$$
\end{ex}

\section{Polynomial families realizing centered polygonal numbers}
\label{sec:centered-polygonal-surfaces}

We now consider the prime-power subsequences of the centered polygonal
numbers introduced in Definition~\ref{def:centered-polygonal-number}. For $m\geq3$,
$$
C_m(q)
=
1+\frac{mq(q-1)}{2}.
$$

By Theorem~\ref{thm q(q-1)+qN_0}, the point count of every
defining family $\mathcal G$ satisfying
\eqref{eq:defining-polynomials} is divisible by $q$.
However, for every odd prime power $q$,
$$
C_m(q)\equiv1\pmod q.
$$
Hence $C_m(q)$ cannot be realized directly as the
$\mathbb F_q$-rational point count of a defining family satisfying
\eqref{eq:defining-polynomials} for every prime power $q$.

We therefore first remove the central point and define
$$
C_m^\ast(n)
=
C_m(n)-1
=
\frac{mn(n-1)}{2}.
$$
Thus
$$
C_m^\ast(q)
=
\frac{mq(q-1)}{2}.
$$
For a fixed integer $m\geq3$, set
$$
\mathcal C_m^{\ast,\mathcal Q}
=
\left(
C_m^\ast(q)
\right)_{q\in\mathcal Q}.
$$

Since $C_m^\ast(2)=m$, the argument of
Corollary~\ref{cor:polygonal-bound} shows that
$C_m^\ast(q)$ can be realized as the point count of a
defining family for every prime power $q$ only if $m$
is even. We therefore first consider even $m$.

\subsection{Even centered polygonal numbers without the central point}
\label{subsec:even-centered-polygonal-surfaces}

Let $m\geq4$ be even. Since
$$
C_m^\ast(q)=\frac{m}{2}q(q-1),
$$
Theorem~\ref{thm q(q-1)+qN_0} gives the following criterion.

\begin{lemma}\label{lem:centered-even-N0}
Let $m\geq4$ be even, and let $\mathcal G$ be a defining
family satisfying \eqref{eq:defining-polynomials}. Then
$$
\#S_{\mathcal G}(\mathbb F_q)=C_m^\ast(q)
$$
if and only if
$$
N_0(\mathcal G)=\frac{m-2}{2}(q-1).
$$
\end{lemma}

\begin{proof}
The result follows directly from
Theorem~\ref{thm q(q-1)+qN_0}.
\end{proof}

We next determine the restriction on $m$ imposed by the
point-count bound.

\begin{prop}\label{prop:centered-q2-suffices}
Let $n\geq2$ and let $m\geq4$ be even. Then
$$
C_m^\ast(q)\leq q(q-1)+q^{n+1}
$$
for every prime power $q$ if and only if
$$
4\leq m\leq2^{n+1}+2.
$$
\end{prop}

\begin{proof}
The proof is similar to that of Proposition~\ref{prop:q2-suffices},
using $2^n(q-1)\leq q^n$ for $n\geq2$.
\end{proof}

Proposition~\ref{prop:centered-q2-suffices} gives the same bound as
Proposition~\ref{prop:q2-suffices}. Hence we use the sets $I_n$ from
Definition~\ref{def:In}, the parent map $\Par$ from
Definition~\ref{def:parent}, and the operators $L$ and $E$ from
Definitions~\ref{def:L-operation} and~\ref{def:extension}. In
particular,
$
I_2=\{4,6,8,10\}.
$

\medskip
Let $m\in I_n$ with $n\geq3$. Since
$
m-\Par(m)=2^n,
$
Lemma~\ref{lem:centered-even-N0} shows that the required increase in
$N_0$ from $\Par(m)$ to $m$ is
$$
\frac{m-\Par(m)}{2}(q-1)
=
2^{n-1}(q-1).
$$
By Lemma~\ref{lem:LE-N0} and
Definition~\ref{def:extension-number}, the recursive step
$E^{\nu(m)}\circ L$ produces exactly this increase in $N_0$,
as in the even-gonal construction.

Thus, only the initial associated families corresponding to
$
m \in I_2=\{4,6,8,10\}
$
need to be changed.

By Lemma~\ref{lem:centered-even-N0}, the required numbers
of common zeros of the initial associated families
for $m=4,6,8,10$ are
$$
q-1,\qquad 2(q-1),\qquad
3(q-1),\qquad 4(q-1),
$$
respectively.

Define
$$
\rho_1(Y_0,Y_1)=Y_0Y_1-1,
\qquad
\rho_2(Y_0,Y_1)=Y_0-1,
$$
$$
\rho_3(Y_0,Y_1)=Y_1,
\qquad
\rho_4(Y_0,Y_1)=Y_1-\Delta(Y_0-1).
$$
For $m\in I_2$, set
$$
\Psi_m(Y_0,Y_1)
=
\prod_{j=1}^{(m-2)/2}\rho_j(Y_0,Y_1),
$$
and define
\begin{equation}
\mathcal F_C^{(m)}
=
\bigl(Q(Y_0),\Psi_m(Y_0,Y_1)\bigr).
\label{eq:initial-centered-Fm}
\end{equation}

\begin{prop}\label{prop:initial-centered-systems}
Let $m\in\{4,6,8,10\}$, and let $\mathcal F_C^{(m)}$ be defined by
\eqref{eq:initial-centered-Fm}. If $\mathcal G_C^{(m)}$ is a defining
family whose associated family is $\mathcal F_C^{(m)}$, then
$$
\#S_{\mathcal G_C^{(m)}}(\mathbb F_q)
=
C_m^\ast(q)
=
\frac{mq(q-1)}{2}.
$$
\end{prop}

\begin{proof}
Since $Q(y_0)=0$ for every $y_0\in\mathbb F_q$,
$
N_0(\mathcal G_C^{(m)})
$
is the number of zeros of
$
\Psi_m(Y_0,Y_1)
$
in $\mathbb F_q^2$.

We consider the four values of $m$.

\begin{enumerate}
\item For $m=4$,
$
\Psi_4(Y_0,Y_1)=Y_0Y_1-1.
$
Its zeros are
$
(u,u^{-1}),
~
u\in\mathbb F_q^\times.
$
Hence
$
N_0(\mathcal G_C^{(4)})=q-1.
$

\item For $m=6$,
$
\Psi_6(Y_0,Y_1)
=
(Y_0Y_1-1)(Y_0-1).
$
The zero sets of the two factors contain $q-1$ and $q$ points,
respectively, and intersect only at $(1,1)$. Therefore
$$
N_0(\mathcal G_C^{(6)})
=
(q-1)+q-1
=
2(q-1).
$$

\item For $m=8$,
$
\Psi_8(Y_0,Y_1)
=
(Y_0Y_1-1)(Y_0-1)Y_1.
$
The zero sets of the three factors contain $q-1$, $q$, and $q$ points, respectively.
The first and second intersect at $(1,1)$, the second and third at
$(1,0)$, and the first and third are disjoint. Hence
$$
N_0(\mathcal G_C^{(8)})
=
(q-1)+q+q-2
=
3(q-1).
$$

\item For $m=10$,
$
\Psi_{10}(Y_0,Y_1)
=
(Y_0Y_1-1)(Y_0-1)Y_1
\bigl(Y_1-\Delta(Y_0-1)\bigr).
$
By the preceding case, the first three factors have $3(q-1)$ zeros.
The last factor vanishes at
$
(1,0)$
and 
$(u,1),
~
u\in\mathbb F_q\setminus\{1\},
$
and hence has $q$ zeros. It is disjoint from the zero set of
$Y_0Y_1-1$ and meets the union of the zero sets of $Y_0-1$ and $Y_1$
only at $(1,0)$. Thus it contributes $q-1$ new zeros, and
$$
N_0(\mathcal G_C^{(10)})
=
3(q-1)+(q-1)
=
4(q-1).
$$
\end{enumerate}

Thus, for every $m\in\{4,6,8,10\}$,
$$
N_0(\mathcal G_C^{(m)})
=
\frac{m-2}{2}(q-1).
$$
The result follows from Lemma~\ref{lem:centered-even-N0}.
\end{proof}

We now use these four associated families as the initial cases and apply
the same recursive step as in
Definition~\ref{def:recursive-even-family}.

\begin{defi}
\label{def:recursive-centered-family}
For every even integer $m\geq12$, define
$$
\mathcal F_C^{(m)}
=
E^{\nu(m)}
\left(
L\bigl(\mathcal F_C^{(\Par(m))}\bigr)
\right).
$$
\end{defi}

By Lemma~\ref{lem:parent-termination}, repeated application of $\Par$
reaches an element of $I_2$. Hence
Definition~\ref{def:recursive-centered-family}, together with the
initial families in \eqref{eq:initial-centered-Fm}, defines
$\mathcal F_C^{(m)}$ for every even $m\geq4$.

The next lemma gives the number of polynomials in these families.

\begin{lemma}
\label{lem:centered-recursive-family-size}
Let $m\in I_n$. Then $\mathcal F_C^{(m)}$ consists of $n$ polynomials.
\end{lemma}

\begin{proof}
The proof is the same as that of
Lemma~\ref{lem:recursive-family-size}, since each initial family
$\mathcal F_C^{(a)}$, $a\in I_2$, consists of two polynomials and
Definition~\ref{def:recursive-centered-family} uses the same recursive
operators.
\end{proof}

We now prove the point-count formula for all even values of $m$.

\begin{theorem}
\label{thm:centered-even-realization}
For every even integer $m\geq4$ and every prime power $q$, there exists
a defining family $\mathcal G_C^{(m)}$ such that
$$
\#S_{\mathcal G_C^{(m)}}(\mathbb F_q)
=
C_m^\ast(q)
=
\frac{mq(q-1)}{2}.
$$
Consequently,
$$
\left(
\#S_{\mathcal G_C^{(m)}}(\mathbb F_q)
\right)_{q\in\mathcal Q}
=
\left(
C_m^\ast(q)
\right)_{q\in\mathcal Q}
=
\mathcal C_m^{\ast,\mathcal Q},
$$
the prime-power subsequence of the centered $m$-gonal numbers without
the central point.
\end{theorem}
\begin{proof}
By \eqref{eq:Gi-FR}, for every $\mathcal F_C^{(m)}$ there exists a
defining family $\mathcal G_C^{(m)}$ having $\mathcal F_C^{(m)}$ as
its associated family.

The remaining argument follows that of
Theorem~\ref{thm:even-polygonal-realization}. The initial cases are
given by Proposition~\ref{prop:initial-centered-systems}, and
Definition~\ref{def:recursive-centered-family} uses the same recursive
step. Hence
$$
N_0(\mathcal G_C^{(m)})
=
\frac{m-2}{2}(q-1).
$$
Therefore, by Lemma~\ref{lem:centered-even-N0},
$$
\#S_{\mathcal G_C^{(m)}}(\mathbb F_q)
=
C_m^\ast(q)
=
\frac{mq(q-1)}{2}.
$$
The equality of the prime-power subsequences follows.
\end{proof}

\subsection{Odd centered polygonal numbers without the central point}
\label{subsec:odd-centered-polygonal-surfaces}

Let $m\geq3$ be odd. Since $m+1$ is even,
Theorem~\ref{thm:centered-even-realization} gives a defining family
$\mathcal G_C^{(m+1)}$ such that
$$
\#S_{\mathcal G_C^{(m+1)}}(\mathbb F_q)
=
C_{m+1}^\ast(q).
$$
Moreover,
$$
C_{m+1}^\ast(q)-C_m^\ast(q)
=
\frac{q(q-1)}{2}.
$$
Thus, as in the odd polygonal-number construction, it is enough to
retain all points with $x=0$ and exactly half of the points with
$x\neq0$. We use the polynomial $\Theta$ and the notation
$S_{\mathcal G}^{\Theta}(\mathbb F_q)$ introduced in
Subsection~\ref{subsec:odd-auxiliary-polynomial}.

\begin{defi}
\label{def:centered-odd-family}
Let $m\geq3$ be odd, and let
$$
\mathcal G_C^{(m+1)}
=
\left(
G_{C,1}^{(m+1)},\ldots,
G_{C,n(m+1)}^{(m+1)}
\right)
$$
be a defining family as in
Theorem~\ref{thm:centered-even-realization}. Define
$$
\mathcal G_C^{(m)}
=
\left(
G_{C,1}^{(m+1)},\ldots,
G_{C,n(m+1)}^{(m+1)},\Theta
\right),
$$
and set
$$
S_{\mathcal G_C^{(m)}}(\mathbb F_q)
=
S_{\mathcal G_C^{(m+1)}}^{\Theta}(\mathbb F_q).
$$
\end{defi}

\begin{theorem}
\label{thm:centered-odd-realization}
For every odd integer $m\geq3$ and every prime power $q$, every
polynomial family $\mathcal G_C^{(m)}$ obtained as in
Definition~\ref{def:centered-odd-family} satisfies
$$
\#S_{\mathcal G_C^{(m)}}(\mathbb F_q)
=
C_m^\ast(q)
=
\frac{mq(q-1)}{2}.
$$
Consequently,
$$
\left(
\#S_{\mathcal G_C^{(m)}}(\mathbb F_q)
\right)_{q\in\mathcal Q}
=
\left(
C_m^\ast(q)
\right)_{q\in\mathcal Q}
=
\mathcal C_m^{\ast,\mathcal Q}.
$$
\end{theorem}

\begin{proof}
The proof is the same as that of
Theorem~\ref{thm:odd-polygonal-realization}. By
Lemma~\ref{lem:theta-half-count}, all points with $x=0$ remain, while
Corollary~\ref{cor:x-nonzero-fiber} and
Lemma~\ref{lem:theta-half-count} show that exactly
$q(q-1)/2$ of the points with $x\neq0$ remain. Hence
$$
\begin{aligned}
\#S_{\mathcal G_C^{(m)}}(\mathbb F_q)
&=
C_{m+1}^\ast(q)-\frac{q(q-1)}{2}\\
&=
C_m^\ast(q).
\end{aligned}
$$
\end{proof}

\subsection{Recovering the central point}
\label{subsec:recovering-central-point}

Let $\mathcal G_C^{(m)}$ be the polynomial family obtained from
Theorem~\ref{thm:centered-even-realization} when $m$ is even and from
Definition~\ref{def:centered-odd-family} when $m$ is odd. Since
$$
C_m(q)=C_m^\ast(q)+1,
$$
we add one point to $S_{\mathcal G_C^{(m)}}(\mathbb F_q)$.

Let $n$ be determined by
$$
\begin{cases}
m\in I_n, & \text{if $m$ is even},\\
m+1\in I_n, & \text{if $m$ is odd}.
\end{cases}
$$For $n\geq3$,
Proposition~\ref{prop:level-of-m} gives
$$
n=
\begin{cases}
\left\lfloor\log_2(m-4)\right\rfloor,
& \text{if $m$ is even},\\
\left\lfloor\log_2(m-3)\right\rfloor,
& \text{if $m$ is odd}.
\end{cases}
$$

The variables of $\mathcal G_C^{(m)}$ are
$X,Y_0,\ldots,Y_n$. 

\smallskip
Introduce a new variable $Z$. We associate to
$\mathcal G_C^{(m)}$ the set
\begin{equation}
\widetilde S_{\mathcal G_C^{(m)}}(\mathbb F_q)
=
\left\{
(z,x,y_0,\ldots,y_n)\in\mathbb F_q^{n+3}:
\begin{array}{l}
z(z-1)=0,\\
(1-z)R(x,y_0,\ldots,y_n)=0
\text{ for all }R\in\mathcal G_C^{(m)},\\
zx=zy_0=\cdots=zy_n=0
\end{array}
\right\}.
\label{eq:centered-complete-set}
\end{equation}

\begin{theorem}
\label{thm:centered-polygonal-realization}
For every integer $m\geq3$ and every prime power $q$, let
$\widetilde S_{\mathcal G_C^{(m)}}(\mathbb F_q)$ be the set defined in
\eqref{eq:centered-complete-set}. Then
$$
\#\widetilde S_{\mathcal G_C^{(m)}}(\mathbb F_q)
=
C_m(q)
=
1+\frac{mq(q-1)}{2}.
$$
Consequently,
$$
\left(
\#\widetilde S_{\mathcal G_C^{(m)}}(\mathbb F_q)
\right)_{q\in\mathcal Q}
=
\left(
C_m(q)
\right)_{q\in\mathcal Q},
$$
the prime-power subsequence of the centered $m$-gonal numbers.
\end{theorem}

\begin{proof}
By \eqref{eq:centered-complete-set}, every point satisfies
$$
z(z-1)=0.
$$
Hence $z=0$ or $z=1$.

\begin{enumerate}
\item[(i)] If $z=0$, then
$$
(1-z)R(x,y_0,\ldots,y_n)=0
$$
reduces to
$$
R(x,y_0,\ldots,y_n)=0
$$
for every $R\in\mathcal G_C^{(m)}$. Hence the points with $z=0$ form
$
\{0\}\times S_{\mathcal G_C^{(m)}}(\mathbb F_q).
$

\item[(ii)] If $z=1$, then
$$
zx=zy_0=\cdots=zy_n=0
$$
gives
$$
x=y_0=\cdots=y_n=0.
$$
Hence the only point with $z=1$ is
$
(1,0,\ldots,0).
$
\end{enumerate}

Therefore
$$
\#\widetilde S_{\mathcal G_C^{(m)}}(\mathbb F_q)
=
\#S_{\mathcal G_C^{(m)}}(\mathbb F_q)+1
=
C_m^\ast(q)+1
=
C_m(q),
$$
where the second equality follows from
Theorem~\ref{thm:centered-even-realization} if $m$ is even and from
Theorem~\ref{thm:centered-odd-realization} if $m$ is odd.
\end{proof}

\section{Symbolic codes for polygonal-number associated families}
\label{sec:symbolic-codes}

The associated families $\mathcal F^{(m)}$ are constructed
recursively from the initial families in
\eqref{eq:initial-Fm} by repeated applications of operators 
$L$ and $E$ (Definition~\ref{def:recursive-even-family}).
For large $m$, this involves writing several intermediate
polynomial families. We introduce symbolic codes to
record the recursive steps and determine the final
associated family without explicitly writing the
intermediate families. We first consider the even
polygonal-number case and then extend the procedure
to the odd and centered cases.

\subsection{Symbolic notation and decoding}
\label{subsec:symbolic-codes}

We first define the polynomials represented by the
symbolic entries. 
Recall $\alpha$ and $\gamma$ from
Subsection~\ref{subsec:initial-associated-families}. Set
\begin{equation}
Q_0(Y_0)=Y_0^q-Y_0,
\label{eq:code-Q0}
\end{equation}
and, for $s\geq1$, define
\begin{equation}
Q_s(Y_0,\ldots,Y_s)
=
\alpha(Y_0,Y_s)
+
\gamma(Y_0)Q_{s-1}(Y_1,\ldots,Y_s).
\label{eq:code-Q-recursion}
\end{equation}

Recall $\beta$ from \eqref{eq:beta}.
Set
\begin{equation}
\beta_1(Y_0,Y_1)
=
\beta(Y_0,Y_1)
=
\alpha(Y_0,Y_1)+\gamma(Y_0)Y_1,
\label{eq:code-beta1}
\end{equation}
and, for $s\geq2$, define
\begin{equation}
\beta_s(Y_0,\ldots,Y_s)
=
\alpha(Y_0,Y_s)
+
\gamma(Y_0)\beta_{s-1}(Y_1,\ldots,Y_s).
\label{eq:code-beta-recursion}
\end{equation}

For $a\in I_2=\{4,6,8,10\}$, recall
$\Omega$ and $\ell_j$ from
\eqref{eq:Omega} and \eqref{eq:ell-j}, respectively.
Set
\begin{equation}
\sigma_{a,1}(Y_0,Y_1)
=
\Omega(Y_0,Y_1)
\prod_{j=0}^{(a-4)/2}\ell_j(Y_0,Y_1).
\label{eq:code-seed-polynomials}
\end{equation}
Thus $\sigma_{a,1}$ is the second polynomial of
$\mathcal F^{(a)}$ in \eqref{eq:initial-Fm}.
For $s\geq2$, define
\begin{equation}
\sigma_{a,s}(Y_0,\ldots,Y_s)
=
\alpha(Y_0,Y_s)
+
\gamma(Y_0)\sigma_{a,s-1}(Y_1,\ldots,Y_s).
\label{eq:code-sigma-recursion}
\end{equation}

We introduce the symbolic notation in
Table~\ref{tab:code-entries}. For an entry
at position $t+1$, the bracketed indices, if any,
specify the intermediate variables, while $Y_t$
is the last variable of the corresponding polynomial.

\begin{table}[htbp]
\centering
\renewcommand{\arraystretch}{1.4}
\begin{tabular}{|c|c|c|}
\hline
\textbf{Position}
&
\textbf{Code entry}
&
\textbf{Polynomial represented}
\\
\hline
$1$
&
$Q_0$
&
$Q_0(Y_0)$
\\
\hline
$t+1$
&
$Q$
&
$Q_1(Y_0,Y_t)$
\\
\hline
$t+1$
&
$Q[1,\ldots,s]$
&
$Q_{s+1}(Y_0,Y_1,\ldots,Y_s,Y_t)$
\\
\hline
$t+1$
&
$\beta$
&
$\beta_1(Y_0,Y_t)$
\\
\hline
$t+1$
&
$\beta[1,\ldots,s]$
&
$\beta_{s+1}(Y_0,Y_1,\ldots,Y_s,Y_t)$
\\
\hline
$t+1$
&
$\sigma_a$
&
$\sigma_{a,1}(Y_0,Y_t)$
\\
\hline
$t+1$
&
$\sigma_a[1,\ldots,s]$
&
$\sigma_{a,s+1}(Y_0,Y_1,\ldots,Y_s,Y_t)$
\\
\hline
\end{tabular}
\caption{Interpretation of symbolic entries in
a code of size $n$, where $1\leq t\leq n-1$,
$a\in I_2$, and $1\leq s<t$ for bracketed entries.}
\label{tab:code-entries}
\end{table}

\begin{defi}
\label{def:symbolic-code}
A symbolic code of size $n\geq2$ is an ordered tuple
\begin{equation}
\C=(c_1,\ldots,c_n),
\label{eq:def-symbolic-code}
\end{equation}
where $c_1=Q_0$ and, for $2\leq i\leq n$,
$c_i$ is an entry from Table~\ref{tab:code-entries}
satisfying the restrictions on its position and indices.

Define the decoding map by
\begin{equation}
\operatorname{Dec}(\C)=(F_1,\ldots,F_n),
\label{eq:code-decoding}
\end{equation}
where $F_i$ is the polynomial represented by $c_i$
in Table~\ref{tab:code-entries}.
\end{defi}

\subsection{Operations on symbolic codes}
\label{subsec:code-operations}

We first determine how the operator $L$ acts on
the polynomials represented by symbolic entries.

\begin{lemma}
\label{lem:symbolic-entry-lift}
Under the operator $L$ of
Definition~\ref{def:L-operation}, the polynomial
$Q_0(Y_0)$ is lifted to $Q_1(Y_0,Y_1)$.

Let $R$ denote $Q$, $\beta$, or $\sigma_a$,
where $a\in I_2$. For $0\leq s<t$, the
polynomial
$
R_{s+1}(Y_0,\ldots,Y_s,Y_t)
$
at position $t+1$ is lifted to
$
R_{s+2}(Y_0,\ldots,Y_{s+1},Y_{t+1})
$
at position $t+2$.
\end{lemma}

\begin{proof}
By Definition~\ref{def:L-operation},
the lift of $Q_0(Y_0)$ is
$$
\alpha(Y_0,Y_1)+\gamma(Y_0)Q_0(Y_1)
=
Q_1(Y_0,Y_1),
$$
by \eqref{eq:code-Q-recursion}.

For the remaining polynomials, the lift of
$R_{s+1}(Y_0,\ldots,Y_s,Y_t)$ is
\begin{align*}
&\alpha(Y_0,Y_{t+1})
+\gamma(Y_0)
R_{s+1}(Y_1,\ldots,Y_{s+1},Y_{t+1})\\
&\qquad=
R_{s+2}(Y_0,\ldots,Y_{s+1},Y_{t+1}),
\end{align*}
by \eqref{eq:code-Q-recursion},
\eqref{eq:code-beta-recursion}, or
\eqref{eq:code-sigma-recursion},
according to the choice of $R$.
\end{proof}

Hence each existing entry moves one position
forward, and its polynomial representation
acquires one additional intermediate variable.
These transformations are summarized in
Table~\ref{tab:code-lift}.

\begin{table}[htbp]
\centering
\renewcommand{\arraystretch}{1.4}
\begin{tabular}{|c|c|}
\hline
\textbf{Original entry}
&
\textbf{Entry after applying $L$}
\\
\hline
$Q_0$ & $Q$\\
\hline
$Q$ & $Q[1]$\\
\hline
$Q[1,\ldots,s]$ & $Q[1,\ldots,s+1]$\\
\hline
$\beta$ & $\beta[1]$\\
\hline
$\beta[1,\ldots,s]$
& $\beta[1,\ldots,s+1]$\\
\hline
$\sigma_a$ & $\sigma_a[1]$\\
\hline
$\sigma_a[1,\ldots,s]$
& $\sigma_a[1,\ldots,s+1]$\\
\hline
\end{tabular}
\caption{Effect of $L$ on existing symbolic
entries, where $s\geq1$ and $a\in I_2$.}
\label{tab:code-lift}
\end{table}

In addition, $L$
introduces $Q_0$ as the first entry.
The operator $E$ adds a new entry $\beta$,
as in Definition~\ref{def:extension}.
We now define the corresponding operations
on symbolic codes.

\begin{defi}
\label{def:code-operations}
Let $\C=(c_1,\ldots,c_n)$ be a symbolic code.
For $1\leq i\leq n$, let $L(c_i)$ denote the
entry specified by Table~\ref{tab:code-lift}.
Define
\begin{equation}
L(\C)
=
\bigl(Q_0,L(c_1),\ldots,L(c_n)\bigr).
\label{eq:code-L}
\end{equation}

Define
$$
E(\C)=(c_1,\ldots,c_n,\beta).
$$
For $\ell\geq1$,
\begin{equation}
E^\ell(\C)
=
\bigl(
c_1,\ldots,c_n,
\underbrace{\beta,\ldots,\beta}_{\ell\text{ entries}}
\bigr),
\label{eq:code-E-ell}
\end{equation}
with $E^0(\C)=\C$.
\end{defi}

\begin{lemma}
\label{lem:code-compatibility}
For every symbolic code $\C$,
\begin{equation}
\begin{aligned}
\operatorname{Dec}(L(\C))
&=L(\operatorname{Dec}(\C)),\\
\operatorname{Dec}(E(\C))
&=E(\operatorname{Dec}(\C)).
\end{aligned}
\label{eq:code-compatibility}
\end{equation}
\end{lemma}

\begin{proof}
The first identity follows from
Lemma~\ref{lem:symbolic-entry-lift}
and \eqref{eq:code-L}.
The second follows from
Definition~\ref{def:extension}, since
$E$ adds $\beta(Y_0,Y_n)$, represented
by $\beta$ at position $n+1$.
\end{proof}

Therefore, the operations on symbolic codes agree with
the corresponding operations on polynomial
families.

\subsection{Recursive codes and correspondence}
\label{subsec:code-construction}

We now assign codes to the associated families
$\mathcal F^{(m)}$ from
Definition~\ref{def:recursive-even-family}.

\begin{defi}
\label{def:recursive-symbolic-codes}
For $a\in I_2=\{4,6,8,10\}$, define
\begin{equation}
\C(a)=(Q_0,\sigma_a).
\label{eq:initial-codes}
\end{equation}

For $m\in I_n$ with $n\geq3$, define recursively
\begin{equation}
\C(m)
=
E^{\nu(m)}
\left(
L\bigl(\C(\Par(m))\bigr)
\right).
\label{eq:recursive-code}
\end{equation}
\end{defi}

By Lemma~\ref{lem:parent-termination},
Definition~\ref{def:recursive-symbolic-codes}
determines $\C(m)$ for every even integer $m\geq4$.

For the parent chain
$$
m=m_t\longmapsto m_{t-1}\longmapsto\cdots
\longmapsto m_1\longmapsto m_0,
$$
where $m_0\in I_2$, iterating
\eqref{eq:recursive-code} gives
\begin{equation}
\C(m)
=
\left(E^{\nu(m_t)}\circ L\right)
\circ\cdots\circ
\left(E^{\nu(m_1)}\circ L\right)
\bigl(\C(m_0)\bigr).
\label{eq:code-parent-recursion}
\end{equation}

\begin{prop}
\label{prop:code-correspondence}
For every even integer $m\geq4$,
\begin{equation}
\operatorname{Dec}(\C(m))
=
\mathcal F^{(m)}.
\label{eq:code-family-correspondence}
\end{equation}
\end{prop}

\begin{proof}
For $a\in I_2$, the result follows from
\eqref{eq:initial-Fm} and \eqref{eq:initial-codes}.

Suppose it holds for $\Par(m)$. By
Lemma~\ref{lem:code-compatibility},
\begin{align*}
\operatorname{Dec}(\C(m))
&=
E^{\nu(m)}
L\left(\operatorname{Dec}(\C(\Par(m)))\right)\\
&=
E^{\nu(m)}
L\left(\mathcal F^{(\Par(m))}\right)\\
&=\mathcal F^{(m)},
\end{align*}
where the last equality follows from
Definition~\ref{def:recursive-even-family}.
The result follows by induction along the parent chain.
\end{proof}

\subsection{Example and further cases}
\label{subsec:code-example}

\begin{ex}
\label{ex:code-358}
Consider $m=358$. By
Example~\ref{ex:recursive-358}, its parent chain is
$$
358\longmapsto102\longmapsto38\longmapsto6,
$$
where
$$
\nu(38)=2,\qquad
\nu(102)=0,\qquad
\nu(358)=1.
$$

Starting from $\C(6)=(Q_0,\sigma_6)$, we obtain
\begin{align*}
\C(38)
&=E^2\bigl(L(\C(6))\bigr)\\
&=\bigl(Q_0,Q,\sigma_6[1],\beta,\beta\bigr).
\end{align*}
Since $\nu(102)=0$,
\begin{align*}
\C(102)
&=L(\C(38))\\
&=\bigl(
Q_0,Q,Q[1],\sigma_6[1,2],
\beta[1],\beta[1]
\bigr).
\end{align*}
Finally,
\begin{align*}
\C(358)
&=E\bigl(L(\C(102))\bigr)\\
&=\bigl(
Q_0,Q,Q[1],Q[1,2],
\sigma_6[1,2,3],
\beta[1,2],\beta[1,2],\beta
\bigr).
\end{align*}

By Proposition~\ref{prop:code-correspondence},
this code determines the associated family
\begin{equation}
\begin{aligned}
\mathcal F^{(358)}
=\bigl(
&Q_0(Y_0),
Q_1(Y_0,Y_1),
Q_2(Y_0,Y_1,Y_2),
Q_3(Y_0,Y_1,Y_2,Y_3),\\
&\sigma_{6,4}(Y_0,Y_1,Y_2,Y_3,Y_4),
\beta_3(Y_0,Y_1,Y_2,Y_5),
\beta_3(Y_0,Y_1,Y_2,Y_6),
\beta_1(Y_0,Y_7)
\bigr).
\end{aligned}
\label{eq:F358-code}
\end{equation}
\end{ex}

\begin{remark}
Let $m\geq3$ be odd. Since $m+1$ is even,
$\C(m+1)$ determines the associated family
$\mathcal F^{(m+1)}$. A corresponding defining
family $\mathcal G^{(m+1)}$ can be chosen by
\eqref{eq:Gi-FR}. By
Definition~\ref{def:odd-family}, the odd
construction is obtained by adjoining the
auxiliary equation $\Theta(X,Y_0)=0$.
Thus no separate code recursion is required
for odd $m$.
\end{remark}

\begin{remark}
The coding procedure extends to the centered
polygonal-number families. Let $\Psi_a$ represent
the second polynomial of $\mathcal F_C^{(a)}$
in \eqref{eq:initial-centered-Fm}, for $a\in I_2$.
Using the same recursive, decoding, and lift rules
as for $\sigma_a$, define
$\C_C(a)=(Q_0,\Psi_a)$ and construct
$\C_C(m)$ by \eqref{eq:recursive-code},
with $\C$ replaced by $\C_C$.
By Definition~\ref{def:recursive-centered-family},
$$
\operatorname{Dec}(\C_C(m))=\mathcal F_C^{(m)}
$$
for every even $m\geq4$.
\end{remark}

\section{Computational resources}
\label{sec:computational-resources}
As $m$ increases, the corresponding polynomial systems become lengthy and difficult to construct and verify by hand. We therefore provide separate SageMath programs for polygonal numbers and for centered polygonal numbers without the central point.

For a given integer $m\geq3$, the appropriate program constructs the corresponding polynomial system and generates the Macaulay2 code used to verify its rational-point count. After executing the code, entering $f(q)$ returns the term indexed by the prime power $q$ in the corresponding $m$-gonal number sequence. Thus, SageMath carries out the construction, while Macaulay2 verifies the resulting point-count formula.

The programs are available at

\begin{center}
\url{https://github.com/Anit-Kuckian/Polygonal-Numbers-Codes.git}
\end{center}

\section{Further directions}
\label{sec:further-directions}

The point-count constructions developed here suggest possible
applications to affine variety codes and related algebraic coding
constructions. We leave the
study of these applications for future work.

\section*{Acknowledgements}
The first author initiated this work while supported by the grant CRG/2022/007047 and later received financial support from the Department of Science and Technology, Government of India, under the WISE Post-Doctoral Fellowship Programme (Reference No.~DST/WISE-PDF/PM-114/2024(G)). The second author carried out part of this work at IIT Gandhinagar as a Sabarmati Bridge Fellow. The third author was also supported by the grant CRG/2022/007047. The authors gratefully acknowledge the support and encouragement received from IIT Gandhinagar.

\section*{Declaration on the Use of AI}
\noindent
The mathematical content of this manuscript, including its ideas, constructions, results, and proofs, was developed by the authors. AI tools were used in a limited supporting role for language refinement and presentation, to assist in generating portions of the SageMath code described in Section~\ref{sec:computational-resources}, and to generate LaTeX code of the figures. All such material was reviewed and verified by the authors.

\bibliographystyle{plain}
\bibliography{ref}

@software{M2,
  author       = {Grayson, D. R. and Stillman, M. E.},
  title        = {{Macaulay2}, a Software System for Research in Algebraic Geometry},
  howpublished = {Available at \url{http://www2.macaulay2.com}},
  url          = {http://www2.macaulay2.com}
}

@article{Gupta2019,
  author  = {Gupta, N. and Sen, S.},
  title   = {On Double {Danielewski} Surfaces and the Cancellation Problem},
  journal = {Journal of Algebra},
  volume  = {533},
  pages   = {25--43},
  year    = {2019},
  doi     = {10.1016/j.jalgebra.2019.05.027}
}

@book{IrelandRosen1990,
  author    = {Ireland, K. and Rosen, M.},
  title     = {A Classical Introduction to Modern Number Theory},
  edition   = {2},
  series    = {Graduate Texts in Mathematics},
  volume    = {84},
  publisher = {Springer},
  address   = {New York},
  year      = {1990},
  doi       = {10.1007/978-1-4757-2103-4}
}

@book{mullen2013handbook,
  title     = {Handbook of Finite Fields},
  author    = {Mullen, Gary L. and Panario, Daniel},
  volume    = {17},
  year      = {2013},
  publisher = {CRC Press},
  address   = {Boca Raton}
}

@book{figurate,
  title={Figurate numbers},
  author={Deza, Michel-marie and Deza, Elena},
  year={2012},
  publisher={World Scientific}
}

@article{ddsfq,
  title={Counting Rational Points on Danielewski and Double Danielewski Surfaces over Finite Fields},
  author={Gupta, Sakshi and Kuckian, Anit and Sengupta, Indranath},
  journal={arXiv preprint arXiv:2605.24821},
  year={2026},
  url={https://arxiv.org/abs/2605.24821}
}

@manual{sagemath,
  Key          = {SageMath},
  Author       = {{The Sage Developers}},
  Title        = {{S}ageMath, the {S}age {M}athematics {S}oftware {S}ystem ({V}ersion 10.9)},
  note         = {{\tt https://www.sagemath.org}},
  Year         = {2026},
}

\end{document}